\documentclass[oneside,english]{amsart}
\usepackage[T1]{fontenc}
\usepackage[latin9]{inputenc}
\usepackage{geometry}
\usepackage{amstext}
\usepackage{amsthm}
\usepackage{amssymb}
\usepackage{graphicx}
\usepackage{esint}

\makeatletter
\numberwithin{equation}{section}
\numberwithin{figure}{section}
\theoremstyle{plain}
\newtheorem{thm}{\protect\theoremname}
\theoremstyle{plain}
\newtheorem{prop}[thm]{\protect\propositionname}
\theoremstyle{plain}
\newtheorem{lem}[thm]{\protect\lemmaname}
\theoremstyle{remark}
\newtheorem{rem}[thm]{\protect\remarkname}

\usepackage{todonotes}
\global\long\def\Re{\operatorname{Re}}

\global\long\def\Im{\operatorname{Im}}

\global\long\def\Arg{\operatorname{Arg}}

\global\long\def\Log{\operatorname{Log}}

\makeatother

\usepackage{babel}
\providecommand{\lemmaname}{Lemma}
\providecommand{\propositionname}{Proposition}
\providecommand{\remarkname}{Remark}
\providecommand{\theoremname}{Theorem}

\begin{document}
\title{On the zeros of certain Sheffer sequences and their cognate sequences}
\author{G.-S. Cheon${}^1$}
\address{${}^1$Department of Mathematics/ Applied Algebra and Optimization Research Center, Sungkyunkwan University, Suwon 16419, Rep. of Korea}
\email{gscheon@skku.edu}
\thanks{ G.-S. Cheon was partially supported by the National Research Foundation of Korea (NRF) grant funded by the Korean government (MSIP) (2016R1A5A1008055 and 2019R1A2C1007518).}
\author{T. Forg\'acs${}^2$}

\address{${}^2$Department of Mathematics, California State University, Fresno, Fresno, CA 93740-8001, USA}
\email{tforgacs@mail.fresnostate.edu}
\email{khangt@mail.fresnostate.edu}
\author{K. Tran${}^2$}

\thanks{The third author thanks the organizers and participants of the workshop on Optimal Point Configurations on Manifolds hosted by the Erwin Schr\"odinger International
Institute for Mathematics and Physics.}

\maketitle
\begin{abstract} Given a Sheffer sequence of polynomials, we introduce the notion of an associated sequence called the cognate sequence. We study the relationship between the zeros of this pair of associated sequences and show that in case of an Appell sequence, as well as a more general family of Sheffer sequences, the zeros of the members of each sequence (for large n) are either real, or lie on a line $\Re z=c$. In addition to finding the zero locus, we also find the limiting probability distribution function of such sequences. \\

\noindent MSC: 05A15, 05A40, 30C15, 30E15\\

\noindent Key words: Sheffer sequence, cognate sequence, zero locus, limiting distribution
\end{abstract}
\section{Introduction}

Sequences of polynomials play a fundamental role in several fields of mathematics, including enumerative combinatorics, functional analysis, applied mathematics, and differential equations. 
Polynomial sequences have been studied extensively from many different points of view \cite{cfkt,Leibman}. Some of the aspects recent research has focused on include their explicit formulas, generating functions, recurrence relations, and zero distributions. 

By a \textit{Sheffer sequence} \cite{Rota, BShap1} we shall mean a polynomial sequence indexed by the nonnegative integers $0,1,2,\ldots$, in which the index of each polynomial equals its degree, satisfying conditions related to the umbral calculus \cite{BRoman} in combinatorics. In this paper, given a Sheffer sequence we introduce the notion of its \textit{cognate sequence}, and study zeros of the cognate sequence of certain Sheffer sequences. This direction of research is largely motivated by polynomial pairs defined using recurrence relations, such as the Lucas polynomial sequences \cite{Cheon, Horadam} for example. \\
A Sheffer sequence $\left\{ G_{n}(s)\right\} _{n=0}^{\infty}$ is characterized by its exponential generating function
\begin{eqnarray*}
\sum_{n=0}^\infty G_n(s){z^n\over n!}=g(z)e^{sf(z)}
\end{eqnarray*}
for some (formal) power series $g$ and $f$ in the variable $z$ satisfying the conditions $g(0)\ne0$, $f(0)=0$ and $f'(0)\ne0$. By convention, we call $\left\{ G_{n}(s)\right\} _{n=0}^{\infty}$ the Sheffer sequence for the pair $(g,f)$. In particular, the Sheffer sequence for a pair $(g,az)$ with a constant $a\ne0$ is called an \textit{Appell sequence}. There are a number of classical polynomial sequences that are Appell sequences, including the Bernoulli polynomials $B_n(s)$ for the pair $({z\over e^z-1},z)$, the Euler polynomials $E_n(s)$ for the pair $({2\over e^z+1},z)$, and the Hermite polynomials $H_n(s)$ for the pair $(e^{-z^2},2z)$.\\
Sheffer sequences form a group called the Sheffer group with the operation of umbral composition, defined as follows (see \cite{Tian}). Suppose
$\left\{ G_{n}(s)\right\} _{n=0}^{\infty}$ and $\left\{ H_{n}(s)\right\} _{n=0}^{\infty}$ are Sheffer sequences for the pairs $(g,f)$ and $(h,\ell)$ respectively, given by
\begin{eqnarray}\label{Sheffer}
G_{n}(s)=\sum_{k=0}^n a_{n,k}s^k\quad{\rm and}\quad H_{n}(s)=\sum_{k=0}^n b_{n,k}s^k.
\end{eqnarray}
Then the umbral composition of $G_{n}(s)$ with $H_{n}(s)$, denoted by $G_n\circ H_n(s)$, is the sequence of polynomials defined by
\begin{eqnarray*}
G_n\circ H_n(s)=\sum_{k=0}^n a_{n,k}H_k(s)=\sum_{0\le\ell\le k\le n} a_{n,k}b_{k,\ell}s^\ell.
\end{eqnarray*}

It is shown in\cite{Tian} that the Sheffer group is isomorphic to the Riordan group of exponential Riordan matrices defined in terms of exponential generating functions as follows.
Let $D=[d_{i,j}]_{i,j\ge0}$ be an infinite lower triangular matrix with complex entries. If there exists a pair of exponential generating functions 
$$
g=\sum_{k\ge0}g_k {z^k\over k!}\;\;{\rm and}\;\;f=\sum_{k\ge1}f_k {z^k\over k!}
$$
with $g_0\ne0$ and $f_1\ne 0$, such that the $k$-th column of $D$ has exponential generating function $gf^k/k!$ for $k=0,1,2,\ldots$, then $D$ is called an {\it exponential Riordan matrix}, and is denoted by $[g,f]$. Let $\mathcal{R}$ be the set of
all exponential Riordan matrices. $\mathcal{R}$ is a group called the {\it (exponential) Riordan group} under usual matrix multiplication. In terms of generating functions the product is expressed by 
\begin{eqnarray}
[g,f][h,\ell]=[gh(f),\ell(f)]
\end{eqnarray}
where $h(f)$ denotes composition of power series with $f(0)=0$.

By definition, we see that the coefficient matrices $[a_{n,k}]$ of $\left\{ G_{n}(s)\right\} _{n=0}^{\infty}$ and $[b_{n,k}]$ of $\left\{ H_{n}(s)\right\} _{n=0}^{\infty}$ in (\ref{Sheffer}) are exponential Riordan matrices $[g,f]$ and $[h,\ell]$ respectively. Moreover, $\left\{ G_n\circ H_n(s)\right\} _{n=0}^{\infty}$ is the Sheffer sequence for the pair $(gh(f),\ell(f))$.

We claim that the Riordan group $\mathcal{R}$ is isomorphic to the group $\mathcal{R}'$ of exponential Riordan matrices of the form $\left[{f^{\prime}/ g},f\right]$. To see this, consider the map $\phi:\mathcal{R}\rightarrow\mathcal{R}'$ given by $\phi([g,f])=\left[f^{\prime}/g,f\right]$. Then
for any $A=[g,f]$ and $B=[h,\ell]$ in $\mathcal{R}$, we have
\begin{align*}
\phi(AB)=\phi([gh(f),\ell(f)])=\left[{f^{\prime}\ell^{\prime}(f)\over
gh(f)},\ell(f)\right]=\left[{f^{\prime}\over g},f\right]\left[{\ell^{\prime}\over
h},\ell\right]=\phi(A)\phi(B).
\end{align*}
Hence $\phi$ is a group homomorphism. In addition, ${\rm
ker}(\phi)=\{(1,z)\}$ and clearly, $\phi$ is onto. Thus, $\phi$ is a group
isomorphism. We may thus associate to the Sheffer sequence $\left\{G_{n}(s)\right\} _{n=0}^{\infty}$ for the pair $(g,f)$, its {\it cognate sequence} $\left\{G_{n}^c(s)\right\} _{n=0}^{\infty}$ generated by the relation
\begin{eqnarray*}
{f'(z)\over g(z)}e^{sf(z)}=\sum_{n\ge0}G_n^c(s){z^n\over n!}.
\end{eqnarray*} 
For each $n$, we call $G_n^c(s)$ the cognate polynomial of $G_n(s)$. It is natural to ask how the zeros of $G_n(s)$ relate to the zeros of the cognate polynomial $G_{n}^c(s)$. After all, the map 
\[
\left\{G_{n}(s)\right\} _{n=0}^{\infty} \stackrel{\Phi}{\longrightarrow} \left\{G_{n}^c(s)\right\} _{n=0}^{\infty}
\]
is a transformation on $\mathbb{R}[x]$, and the properties of such transformations, as they relate to the preservation of zero locus, have been a central problem of study in the context of the P\'olya-Schur program (see \cite{BB}) and beyond.  
In this paper we study a subset of such maps -- or equivalently pairs of Sheffer sequences and their cognate sequence -- which preserve the \textit{symmetry type} of the zero locus of a Sheffer sequence.  \\
The paper is organized as follows. In Section \ref{sec:appellzeros} we discuss Appell sequences and their cognate sequences, building on the example of the Bernoulli polynomials. We also provide a characterization of all Appell sequneces whose zeros exhibit the same type of symmetry as those of the Bernoulli polynomials.  In Section \ref{sec:sheffercognate} we show that the Sheffer sequences considered in \cite{cfkt} along with their cognate sequence are generated by a pair of functions of the same general form. We show that any Sheffer sequence generated by functions of this kind consist of polynomials $H_n$ whose zeros are either real or lie on a line of the form $\Re z=c$ for $n \gg 1$. We accomplish this in two subsections: the first (subsection \ref{subsec:asymptotics}) develops the necessary asymptotic formulas for the integral representation of the polynomials under investigation, while the second (subsection \ref{subsec:zerolocation}) finds the precise location of the zeros of this sequence. The paper concludes with Section \ref{sec:limitdist}, in which we discuss the limiting distribution of the zeros of the family of sequences defined in Theorem \ref{thm:maintheorem}.
\section{The zeros of Appell sequences and their cognate sequences}\label{sec:appellzeros}
We begin our investigations with the zeros of the cognate sequence of the Bernoulli polynomials. By definition, the cognate sequence $\left\{B_{n}^c(s)\right\} _{n=0}^{\infty}$ of the Bernoulli polynomials $\left\{B_{n}(s)\right\} _{n=0}^{\infty}$
is the Appell sequence for the pair $({e^z-1\over z},z)$. It is known that all the zeros of Bernoulli polynomials $B_n(s)$ $(n\ge1)$ are symmetrical with respect to the line $\Re s={1\over2}$.

Our first theorem (c.f. Theorem \ref{Bcog}) demonstrates that the location of the zeros of $B_n^c(s)$ is closely related to that of the zeros of $B_n(s)$. In the proof of this result we need to make use of the following lemma. 

\begin{lem}\label{realzero} \cite{BBump,BTitch}
Let $G(s)\in\mathbb{C}[s]$ be a polynomial all of whose zeros have
positive imaginary part, and let $\bar{G}(s)$ be the polynomial
whose coefficients are the complex conjugates of those of $G(s)$.
Then all zeros of $G(s)+\bar{G}(s)\in\mathbb{R}[s]$ are real.
\end{lem}

\begin{thm}\label{Bcog}
For $n\ge1$ let $B_n^c(s)$ be the cognate polynomial of the Bernoulli polynomial $B_n(s)$. Then all zeros of $B_n^c(s)$ lie on the line $\Re s=-{1\over2}$. 
\end{thm}
\begin{proof} Define
$G_n(s)=B_n^c\left(-{1\over2}+is\right)$. Then all zeros of
$G_n(s)$ are real if and only if the real part of every zero of
$B_n^c(s)$ is $-{1\over2}$, or equivalently, all zeros of $B_n^c(s)$ lie
on the line $\Re s=-{1\over2}$. Thus it suffices to show that all 
zeros of $G_n(s)$ are real. By the definition of $G_n(s)$ we have
\begin{eqnarray*}
\sum_{n\ge0} G_n(s){z^n\over n!}={e^z-1\over
z}e^{\left(-{1\over2}+is\right)z}={e^{{1\over2}z}-e^{-{1\over2}z}\over
z}e^{isz}={1\over
z}\left(e^{\left({1\over2}+is\right)z}+\left(-e^{\left(-{1\over2}+is\right)z}\right)\right).
\end{eqnarray*}
Let
\begin{eqnarray*}
e^{\left({1\over2}+is\right)z}=\sum_{n\ge0}f_n(s){z^n\over n!}\quad{\rm and}\quad -e^{\left(-{1\over2}+is\right)z}=\sum_{n\ge0}g_n(s){z^n\over
n!}.
\end{eqnarray*}
Since $f_n(s)=\left({1\over2}+is\right)^n$, all zeros of
$f_n(s)$ (and also $-f_n(s)$) have positive imaginary part, namely
${1\over2}$. It is easy to see that $g_n(s)=(-1)^{n+1}\bar{f}_n(s)$.
Hence by Lemma \ref{realzero}, we obtain that all zeros of
$G_n(s)=f_n(s)+(-1)^{n+1}\bar{f}_n(s)$ are real, which completes the
proof.
\end{proof}

The above connection between the zeros of Bernoulli polynomials and their cognate sequence does not extended to arbitrary Appell sequences and their cognate sequences. Thus, one is naturally led to the problem of finding and characterizing all Sheffer sequences and their cognate sequence whose zeros exhibit symmetries akin to that displayed by the zeros of $\left\{B_{n}(s)\right\} _{n=0}^{\infty}$ and $\left\{B_{n}^c(s)\right\} _{n=0}^{\infty}$. Generally, it would be of interest to understand the relationship between the zeros of a Sheffer sequence and its cognate sequence.

To obtain some information about the zeros of the cognate sequences of Appell sequences, we begin with the following lemma. 

\begin{lem}\label{symzero}
Let $G(s)\in\mathbb{R}[s]$ with degree $n\ge1$. Then all
zeros of $G(s)$ are symmetrical with respect to the line
$\Re s=-{m\over2}$ $(m\in\mathbb{R})$ if and only if
$G(-s)=(-1)^nG(s-m)$.
\end{lem}
\begin{proof} Suppose that all zeros of $G(s)$ are
symmetrical with respect to the line $\Re s=-{m\over2}$ for some
$m\in\mathbb{R}$. First let $n$ be even. Then we may assume that $n$
zeros of $G(s)$ are of the form $-{m\over2}\pm q_k$
$(q_k\in\mathbb{C}, k=1,2,\ldots,{n\over2})$ so that $G(s)$ can be
written as
\begin{eqnarray*}
G(s)=\prod_{k=1}^{n/2}\left(s+{m\over2}+q_k\right)\left(s+{m\over2}-q_k\right).
\end{eqnarray*}
Hence
\begin{eqnarray*}
G(-s)&=&\prod_{j=1}^{n/2}\left(-\left(s-{m\over2}-q_k\right)\right)\left(-\left(s-{m\over2}+q_k\right)\right)
\\&=&(-1)^n\prod_{k=1}^{n/2}\left((s-m)+{m\over2}-q_k\right)\left((s-m)+{m\over2}+q_k\right)
=(-1)^nG(s-m).
\end{eqnarray*}
Now let $n$ be odd. Then $G(s)$ is of the form
\begin{eqnarray*}
G(s)=\left(s+{m\over2}\right)^j\prod_{k=1}^{(n-j)/2}\left(s+{m\over2}+q_k\right)\left(s+{m\over2}-q_k\right)
\end{eqnarray*}
where $j\ge1$ and $j$ is odd. A simple computation shows that
$G(-s)=(-1)^nG(s-m)$ also holds for this case.

Conversely, suppose that $G(-s)=(-1)^nG(s-m)$ holds for some
$m\in\mathbb{R}$. Let $a+bi$ $(a,b\in\mathbb{R})$ be a zero of
$G(s)$. Then $a-bi=-{m\over2}+((a+{m\over2})-bi)$ is also a zero of
$G(s)$. It follows from
\begin{eqnarray*}
0=G(a-bi)=G\left(-{m\over2}+\left(\left(a+{m\over2}\right)-bi\right)\right)
=(-1)^nG\left(-{m\over2}-\left(\left(a+{m\over2}\right)-bi\right)\right)
\end{eqnarray*}
that $-(a+m)+b_i=-{m\over2}-\left(\left(a+{m\over2}\right)-bi\right)$ is a zero
of $G(s)$, which implies that all zeros of $G(s)$ are
symmetrical with respect to the line $\Re s=-{m\over2}$. This completes
the proof.
\end{proof}

\begin{thm}\label{Appell1}
Let $\{G_n(s)\}_{n\ge0}$ be the Appell sequence for the pair $(g,az)$.
Then the following are equivalent:
\begin{itemize}
\item[{\rm (i)}] For $n\ge1$, the zeros of $G_n(s)$ are symmetric with
respect to the line $\Re s=-{g'(0)\over2a}$ ;
\item[{\rm (ii)}] For $n\ge1$, $G_n(-s)=(-1)^nG_n(s-{g'(0)\over a})$ ;
\item[{\rm (iii)}] $g(z)=g(-z)e^{g'(0)z}$.
\end{itemize}
\end{thm}
\begin{proof} It follows from Lemma \ref{symzero} that
(i) and (ii) are equivalent. Moreover, (ii) holds if and only if
\begin{eqnarray*}
g(z)e^{-asz}&=&\sum_{n\ge0}G_n(-s){z^n\over
n!}=\sum_{n\ge0}(-1)^nG_n\left(s-{g'(0)\over a}\right){z^n\over
n!}=g(-z)e^{-a(s-{g'(0)\over a})z}\\&=&g(-z)e^{g'(0)z}e^{-asz},
\end{eqnarray*}
which is equivalent to (iii).
\end{proof}

\begin{thm}\label{Appell2}
Let $\{G_n(s)\}_{n\ge0}$ be the Appell sequence for the pair $(g,az)$. If
all zeros of $G_n(s)$ for $n\ge1$ are symmetrical with respect to the line
$\Re s=-{g'(0)\over2a}$, then all zeros of $G_n^c(s)$ are
symmetrical with respect to the line $\Re s={g'(0)\over2a}$.
\end{thm}
\begin{proof} It suffices to show that
$G_n^c(-s)=(-1)^nG_n^c(s+{g'(0)\over a})$ holds for all $n\ge1$. By
Theorem \ref{Appell1} we have
\begin{eqnarray*}
\sum_{n\ge0}G_n^c(-s){z^n\over n!}&=&{a\over g(z)}e^{-asz}={a\over
g(-z)}e^{(-as-g'(0))z}={a\over g(-z)}e^{-a(s+{g'(0)\over
a})z}\\&=&\sum_{n\ge0}(-1)^nG_n^c\left(s+{g'(0)\over
a}\right){z^n\over n!},
\end{eqnarray*}
which implies that for all $n \geq 1$, $G_n^c(-s)=(-1)^nG_n^c(s+{g'(0)\over
a})$. The proof is complete.
\end{proof}

 The following theorem shows that there are
infinitely many Appell sequences satisfying the assumptions of Theorem \ref{Appell2}.

\begin{thm}
Let $\{G_n(s)\}_{n\ge0}$ be the Appell sequence for the pair $(g,az)$. If
\begin{eqnarray}\label{e:gtan}
g(z)=2 \rho(z)(1+{\rm tanh}(kz)),
\end{eqnarray}
where $\rho(z)$ is any even function with $\rho(0)=1$ and
$k\in\mathbb{R}$, then all the zeros of $G_n(s)$ are symmetrical with
respect to the line $\Re s=-{k\over a}$ for $n\ge1$. Conversely, if all
the zeros of $G_n(s)$ $(n\ge1)$ are symmetrical with respect to a
vertical line in $\mathbb{C}$ then there exist an even function
$\rho(z)$ with $\rho(0)=1$ and $k\in\mathbb{R}$ satisfying
(\ref{e:gtan}).
\end{thm}
\begin{proof} Suppose $g(z)=2\rho(z)(1+{\rm tanh}(kz))$
for some even function $\rho(z)$ such that $\rho(0)=1$ and
$k\in\mathbb{R}$. Then 
\begin{eqnarray*}
g(-z)e^{2kz}&=&2\rho(z)(1+{\rm
tanh}(-kz))e^{2kz}=2\rho(z)\left(1+{e^{-kz}-e^{kz}\over
e^{-kz}+e^{kz}}\right)e^{2kz}\\&=&2\rho(z)\left(1+{e^{kz}-e^{-kz}\over
e^{kz}+e^{-kz}}\right)=2\rho(z)(1+{\rm tanh}(kz))=g(z),
\end{eqnarray*}
and $g'(0)=2k$. Thus, Theorem \ref{Appell1} (iii) holds, and hence so does (i), i.e. the zeros of $G_n(s)$ $(n\ge1)$
are symmetric with respect to the line $\Re s=-{k\over a}$.
Conversely, suppose that all the zeros of $G_n(s)$ are symmetric
with respect to a vertical line in $\mathbb{C}$. Let
$g(z)=2\left(1+\sum_{n\ge1}g_nz^n\right)$. Since $G_1(s)=2(g_1+as)$, this vertical line is $\Re s=-\frac{g_1}{a}$. By Theorem \ref{Appell1}, $g(z)$ satisfies
$g(z)=g(-z)e^{2g_1z}$. A simple computation shows that
\begin{eqnarray*}
g(z)={g(z)+g(-z)\over2}(1+\tanh(g_1z)).
\end{eqnarray*}
Setting $\rho(z)={g(z)+g(-z)\over2}$ yields an even function, and
$k=g_1\in\mathbb{R}$, as desired.
\end{proof}

\begin{rem}We note that if $\{G_n(s)\}_{n\ge0}$ is the Appell sequence for the pair
$(g(z),z)$ then the Appell sequence for the pair $(g(z),az)$ is
$\{G_n(as)\}_{n\ge0}$. Thus it suffices to explore the Appell sequence for the pair $(g(z),z)$ when studying the zeros of the Appell sequence for the pair $(g(z),az)$.
\end{rem}
 \section{The zeros of a certain family of Sheffer sequences and their cognate sequences}\label{sec:sheffercognate}
We now turn our attention to the cognate sequences of Sheffer sequences previously treated in \cite{cfkt}. Note -- as a preview -- that the symmetry of the zeros of the cognate sequence about a line remains, and that our main result (c.f. Theorem \ref{thm:maintheorem}) also exploits the fact that (a suitable modification of) the non-exponential factor of the generating function is even. \\
In order to set up the statement of the main result, suppose $z_{2}>z_{1}>0$, and let $\Log(\cdot)$ denote the principle logarithm. Set
\begin{align*}
f(z) & =\Log(z_{1}-z)+\Log(z_{2}-z)-\Log(z_{1}+z)-\Log(z_{2}+z)\\
g(z) & =(z_{1}+z)(z_{2}+z).
\end{align*}
The sequence of Sheffer polynomials $\left\{ G_{n}(s)\right\} _{n=0}^{\infty}$
for the pair $(g(z),f(z))$ is generated by 
\[
\sum_{n=0}^{\infty}G_{n}(s)\frac{z^{n}}{n!}=g(z)e^{sf(z)}=(z_{1}-z)^{s}(z_{2}-z)^{s}(z_{1}+z)^{1-s}(z_{2}+z)^{1-s}.
\]
The corresponding cognate sequence $\left\{ G_{n}^{c}(s)\right\} _{n=0}^{\infty}$
is the Sheffer sequence for the pair $(f'(z)/g(z),f(z))$, generated
by 
\[
\sum_{n=0}^{\infty}G_{n}^{c}(s)\frac{z^{n}}{n!}=\frac{f'(z)}{g(z)}e^{sf(z)},
\]
where 
\begin{align*}
\frac{f'(z)}{g(z)} & =-\frac{1}{(z_{1}+z)(z_{2}+z)}\left(\frac{1}{z_{1}-z}+\frac{1}{z_{2}-z}+\frac{1}{z_{1}+z}+\frac{1}{z_{2}+z}\right)\\
 & =-\frac{2(z_{1}+z_{2})}{(z_{1}+z)^{2}(z_{2}+z)^{2}(z_{1}-z)(z_{2}-z)}(z_{1}z_{2}-z^{2}).
\end{align*}
The reader will note that both $g$ and $\frac{f'}{g}$ are of the form 
\[
(z_{1}-z)^{p}(z_{1}+z)^{p^{*}}(z_{2}-z)^{q}(z_{2}+z)^{q^{*}}\prod_{i=1}^{m}(\alpha_{i}-z^{2})^{p_{i}}
\]
for appropriate values of the constants. As Theorem \ref{thm:maintheorem} shows, both the Sheffer sequence for the pair $(g,f)$, and its cognate sequence have zeros that are symmetric about a line $\Re z=k$. This fact about the Sheffer sequence was already addressed in \cite[Theorem 5]{cfkt}. Theorem \ref{thm:maintheorem} is a generalization of that result.
\begin{thm}
\label{thm:maintheorem}Suppose $m\in\mathbb{N}$ and $p,p^{*},q,q^{*}$,$\alpha_{i},p_{i}$,
$1\le i\le m$, are real numbers such that $\alpha_{i}>z_{1}^{2}$.
Let 
\begin{equation}
h(z)=(z_{1}-z)^{p}(z_{1}+z)^{p^{*}}(z_{2}-z)^{q}(z_{2}+z)^{q^{*}}\prod_{i=1}^{m}(\alpha_{i}-z^{2})^{p_{i}}\label{eq:hzdef},
\end{equation}
let
\[
f(z)=\Log(z_{1}-z)+\Log(z_{2}-z)-\Log(z_{1}+z)-\Log(z_{2}+z),
\]
and $\left\{ H_{n}(s)\right\} _{n=0}^{\infty}$ be the Sheffer sequence
for the pair $(h(z),f(z))$. Assume that $p^{*}-p=q^{*}-q:=2c$. If
$c+p<0$, then all the zeros of $H_{n}(s)$, $n\gg1$, lie on the
line $s=c+it$. If $c+p\ge0$, then the same conclusion holds except
for $2\left\lceil c+p\right\rceil $ real zeros, each of which approaches
$c\pm(c+p+1-k)$, $0<k\le\left\lceil c+p\right\rceil $as $n\rightarrow\infty$. 
\end{thm}
\begin{figure}[h]
\begin{centering}
\includegraphics[scale=0.5]{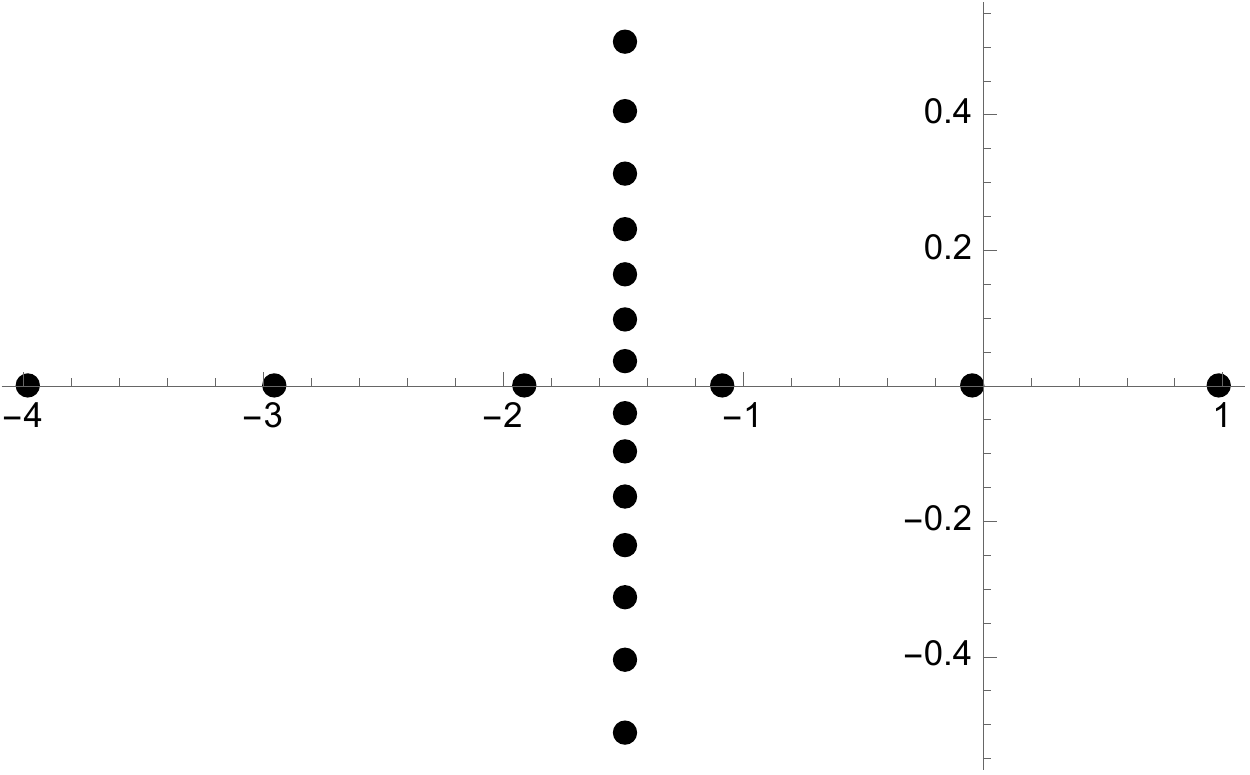} $\qquad$\includegraphics[scale=0.5]{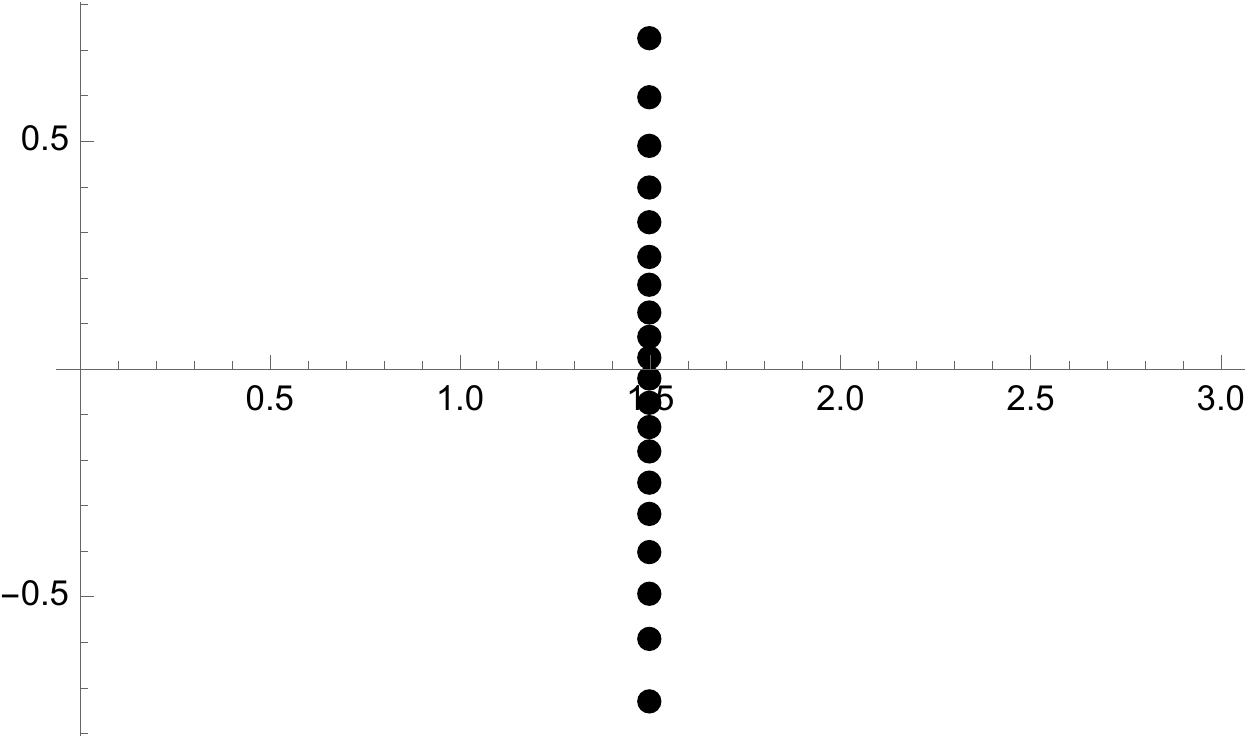}
\par\end{centering}
\caption{Z\label{fig:ZerosHn}eros of $H_{n}(s)$}
\end{figure}
\textit{Example 1.}
Let $z_{1}=1$, $z_{2}=7$. The left side graph in Figure \ref{fig:ZerosHn} shows the zeros of $H_{20}(s)$ when $p=4, p^*=1, q=2, q^*=-1$, $p-p^*=3=q-q^*$, $c+p=\frac{11}{2}$ and
\[
h(z)=(1-z)^{4}(1+z)(7-z)^{2}(7+z)^{-1}(2-z^{2})^{-1}.
\]
The right side graph shows the zeros of $H_{20}(s)$ when $p=-4, p^*=-1, q=-2, q^*=1$, $p-p^*=-3=q-q^*$, $c+p=-\frac{11}{2}$ and

\[
h(z)=(1-z)^{-4}(1+z)^{-1}(7-z)^{-2}(7+z)^{1}(2-z^{2})^{-1}.
\]

%The case $h(z)=g(z)$ and $c=1/2$ satisfies the assumptions of the theorem and is treated completely in \cite{cfkt} .The same conclusion holds for $h(z)=f'(z)/g(z)$ and $c=-1/2$.
The proof of Theorem \ref{thm:maintheorem} is presented in the next two sections, the first of which develops the asymptotics for an integral representation of the $H_n(s)$s, followed by a section on the counting of the zeros of these polynomials on the designated locus.
\subsection{The asymptotic formulas} \label{subsec:asymptotics}
In this section we find an integral representation for the Sheffer polynomials $H_n(c+int)$ described in Theorem \ref{thm:maintheorem}, and we develop of an asymptotic formula for said integral representation, which is uniform on the parameter range $e^{-\ln^4 n} / n \ll t\ll \ln^4 n/n$. To begin, let $\{H_n\}_{n=0}^{\infty}$ be the Sheffer sequence for the pair $(h,f)$ as in the statement of Theorem \ref{thm:maintheorem}. The substitution $s=c+int$ and the Cauchy integral formula gives
\begin{align*}
H_{n}\left(c+int\right) & =\frac{n!}{2\pi i}\ointctrclockwise_{|z|=\epsilon}\frac{h(z)e^{(c+int)f(z)}}{z^{n+1}}dz\\
 & =\frac{n!}{2\pi i}\ointctrclockwise_{|z|=\epsilon}\psi(z)e^{-n\phi(z,t)}dz,
\end{align*}
where 
\[
\phi(z,t)=\Log z-it\left(\Log(z_{1}-z)+\Log(z_{2}-z)-\Log(z_{1}+z)-\Log(z_{2}+z)\right),
\]
and
\[
\psi(z) =\frac{h(z)}{z}\frac{(z_{1}-z)^{c}(z_{2}-z)^{c}}{(z_{1}+z)^{c}(z_{2}+z)^{c}}.
\]
It follows from the definition of $h(z)$ that, as a function of $z$, $\psi(z)e^{-n\phi(z,t)}$
is analytic on the complement of 
\[
\{0\}\cup[z_{1},\infty)\cup(-\infty,-z_{1}].
\]
The defintions of $h(z)$ and $f(z)$ also imply that on any
circle arc $\mathcal{C}_{R}$ with large radius $R$, 
\[
\lim_{R\rightarrow\infty}\int_{\mathcal{C}_{R}}\frac{h(z)e^{(c+int)f(z)}}{z^{n+1}}dz=0, \qquad \textrm{for} \quad n \gg 1.
\]
Thus,
\begin{equation} \label{eq:Hintegral}
H_{n}\left(c+int\right)=\frac{n!}{2\pi i}\int_{\Gamma_{1}\cup\Gamma_{2}}\psi(z)e^{-n\phi(z,t)}dz,
\end{equation}
where $\Gamma_{1}$ and $\Gamma_{2}$ are two halves of a loop around infinity and the cuts
$(-\infty,-z_{1}]$ and $[z_{1},\infty)$ with the counter clockwise
orientation. \\
\begin{figure}[h]
\begin{center}
\includegraphics[width=4 in]{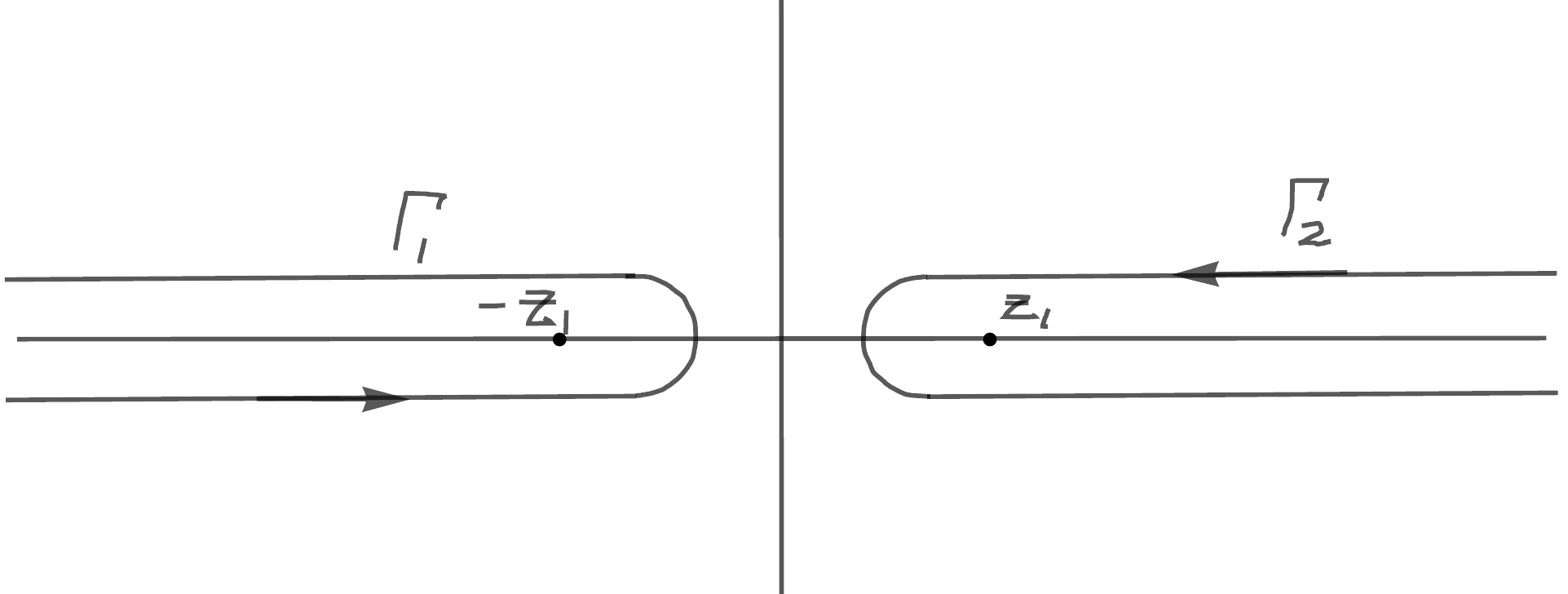}
\caption{The loop around the cuts and infinity}
\label{default}
\end{center}
\end{figure}
The assumption $p^{*}-p=q^{*}-q$ implies that 
\[
h(z)\frac{(z_{1}-z)^{c}(z_{2}-z)^{c}}{(z_{1}+z)^{c}(z_{2}+z)^{c}}
\]
is even. Using the substitution $z \mapsto -z$ we see that the part of the
integral in \eqref{eq:Hintegral} over $\Gamma_{1}$ is equal to
\[
(-1)^{n}\int_{\Gamma_{2}}\psi(z)e^{-n\phi(z,-t)}dz.
\]
Making the subsequent substitution $z \mapsto \overline{z}$, this integral
becomes the conjugate of 
\[
(-1)^{n+1}\int_{\Gamma_{2}}\psi(z)e^{-n\phi(z,t)}dz.
\]
We deduce that $\pi H_{n}(c+int)$ is imaginary part, or $-i$ times
the real part of the integral
\begin{equation}
\int_{\Gamma_{2}}\psi(z)e^{-n\phi(z,t)}dz, \label{eq:intloopcut}
\end{equation}
depending on whether $n$ is even or odd.\\
For the sake of completeness we now present the setup and definitions developed in \cite{cfkt} in order to help establish the asymptotic formula we see. To this end, let 
\[
T_{1}:=\frac{z_{2}-z_{1}}{z_{1}+z_{2}}, \qquad T_{2}:=\frac{z_{1}+z_{2}}{4\sqrt{z_{1}z_{2}}},
\]
and 
\begin{align}
T & =\begin{cases}
T_{1} & \text{if }z_{1}^{2}-6z_{1}z_{2}+z_{2}^{2}\ge0\\
T_{2} & \text{if }z_{1}^{2}-6z_{1}z_{2}+z_{2}^{2}<0.
\end{cases}\label{eq:Tdefn}
\end{align}
For $t\in(0,T)$, the function 
\begin{align}
\zeta=\zeta(t): & =\frac{z_{1}+z_{2}}{2}\left(it-\sqrt{T_{1}^{2}-t^{2}}+\sqrt{1-2t^{2}-2it\sqrt{T_{1}^{2}-t^{2}}}\right)\quad\text{for}\quad0\le t<T_{1},\label{eq:zeta1defnT1}\\
\zeta=\zeta(t): & =\frac{z_{1}+z_{2}}{2}\left(it+i\sqrt{t^{2}-T_{1}^{2}}+\sqrt{1-2t^{2}-2t\sqrt{t^{2}-T_{1}^{2}}}\right)\quad\text{for}\quad T_{1}\le t\le T_{2},\label{eq:zeta1defnT2}
\end{align}
is a solution of $\phi_{z}(z,t)=0$.

Set 
\begin{eqnarray}
g(\zeta)&:=&\frac{2\pi\psi^{2}(\zeta)e^{-2n\phi(\zeta,t)}}{n\phi_{z^{2}}(\zeta,t)},\label{eq:gzetadef} \\
p(\zeta)&:=&\int_{\Gamma_{2}}\psi(z)e^{-n\phi(z,t)}dz \nonumber,
\end{eqnarray}
and consider the following intervals 
\begin{align*}
I_{1} & =\left\{ t|\ln^{4}n/n\ll t<T-\ln^{2}n/n^{2/3}\right\} ,\\
I_{2} & =\left\{ t|\ln^{4}n/n\ll t<T_{1}-\ln^{2}n/n^{2/3}\right\}, \\
I_{3} & =[T_{1}+\ln^{2}n/n^{2/3},T_{2}-\ln^{2}n/n^{2/3}],\\
I & =\left\{ t|1/n^{2/3}\ll T-t<\ln^{2}n/n^{2/3}\right\} .
\end{align*}
With these definitions, the results in \cite{cfkt} show that as $n\rightarrow\infty$, 
\[
p^{2}(\zeta)\sim g(\zeta)
\]
uniformly on $t\in I_{2}\cup I_{3}$ if $T=T_{2}$, and on $t\in I_{1}$
if $T=T_{1}$. Furthermore, if $t\in I$, then
\[
p(\zeta)\sim\frac{4c\psi(\zeta(T))e^{-n\phi(\zeta,t)}}{\sqrt{6}\sqrt{C^{3}\phi_{z^{3}}(\zeta(T),T)}}\alpha_{n}(t)
\]
uniformly in $t$, where 
\[
C=\begin{cases}
\frac{z_{1}+z_{2}}{2}\left(-\sqrt{2T_{1}}+\frac{T_{1}\sqrt{2T_{1}}}{\sqrt{2T_{1}^{2}-1}}\right) & \text{ if }T=T_{1}\\
\frac{\sqrt{32}(z_{1}z_{2})^{3/4}}{\sqrt{(z_{1}+z_{2})(-z_{1}^{2}+6z_{1}z_{2}-z_{2}^{2})}} & \text{ if }T=T_{2}
\end{cases}
\]
and $\Re\alpha_{n}(t)\ge0$. If $T=T_{2}$, then 
\[
p(\zeta)\sim\frac{4d\psi(\zeta(T_{1}))e^{-n\phi(\zeta,t)}}{\sqrt{6}\sqrt{D^{3}\phi_{z^{3}}(\zeta(T_{1}),T_{1})}}\beta_{n}(t)
\]
uniformly on $\frac{1}{n}\ll|T_{1}-t|\le\ln^{2}n/n^{2/3}$, where 
\[
D=-\frac{(z_{1}+z_{2})\sqrt{T_{1}}}{\sqrt{2}}\left(1+\frac{iT_{1}}{\sqrt{1-2T_{1}^{2}}}\right)
\]
and $\Re\beta_{n}(t)\ge0$.
Before we state and prove our main estimate, we note that \eqref{eq:hzdef} implies that as $z\rightarrow z_{1}$, 
\[
h(z)=(z_{1}-z)^{p}\left(h_{p}+\mathcal{O}\left(z_{1}-z\right)\right)
\]
for some $h_{p}\in\mathbb{R}$. In the remainder of this section,
we prove the following asymptotic equivalence. 
\begin{prop}
\label{prop:asympsmallt} Let $h_p$ and $p(\zeta)$ be as defined in the beginning of the section, and let $\Gamma$ denote the Gamma function. As $n\rightarrow\infty$ the following asymptotic formula holds uniformly on $e^{-\ln^{4}n}/n\ll t\ll\ln^{4}n/n$:
\[
p(\zeta)\sim\frac{2i\pi h_{p}(z_{2}-z_{1})^{c+int}}{z_{1}^{n-p}2^{c+int}(z_{2}+z_{1})^{c+int}n^{c+p+1+int}\Gamma(-c-p-int)}.
\]
\end{prop}
\begin{proof}
We rewrite $p(\zeta)$ as 
\[
p(\zeta)=\int_{+\infty}^{(z_{1}^{-})}\psi(z)e^{-n\phi(z,t)}dz,
\]
where path of integration is the Hankel contour which loops around
the ray $[z_{1},\infty)$. The notation $z_{1}^{-}$ means the path
goes around $z_{1}$ in the negative direction. We make the substitution
$w=z/z_{1}$ followed by the subsitution $e^{z}=w$ to arrive at the expression 
\[
p(\zeta(t))=z_{1}\int_{+\infty}^{(1^{-})}\psi(wz_{1})e^{-n\phi(wz_{1},t)}dw=z_{1}\int_{+\infty}^{(0^{-})}\psi(z_{1}e^{z})e^{-n\phi(z_{1}e^{z},t)}e^{z}dz.
\]
Suppose $\epsilon>0$ is small such that $n\epsilon=o(1)$. We break
the last integral into three pieces: 
\begin{align}
 & z_{1}\int_{(0^{-})}^{\left(ln^{5}n/n\right)\pm i\epsilon}\psi(z_{1}e^{z})e^{-n\phi(z_{1}e^{z},t)}e^{z}dz\nonumber \\
+ & z_{1}\int_{\left(\ln^{5}n/n\right)+i\epsilon}^{+\infty+i\epsilon}\psi(z_{1}e^{z})e^{-n\phi(z_{1}e^{z},t)}e^{z}dz\nonumber \\
+ & z_{1}\int_{+\infty-i\epsilon}^{\left(\ln^{5}n/n\right)-i\epsilon}\psi(z_{1}e^{z})\frac{2i\pi h_{p}(z_{2}-z_{1})^{c+int}}{z_{1}^{n-p}2^{c+int}(z_{2}+z_{1})^{c+int}n^{c+p+1+int}\Gamma(-c-p-int)}e^{-n\phi(z_{1}e^{z},t)}e^{z}dz.\label{eq:splitintegral}
\end{align}
Recall that
\begin{equation}
\psi(z)e^{-n\phi(z,t)}=\frac{h(z)}{z^{n+1}}\frac{(z_{1}-z)^{c}(z_{2}-z)^{c}}{(z_{1}+z)^{c}(z_{2}+z)^{c}}\left(\frac{(z_{1}-z)(z_{2}-z)}{(z_{1}+z)(z_{2}+z)}\right)^{nit},\label{eq:integrand}
\end{equation}
and 
\[
h(z)=(z_{1}-z)^{p}\left(h_{p}+\mathcal{O}\left(z_{1}-z\right)\right)
\]
for $z_{1}-z=o(1)$. Thus, the first integral in \eqref{eq:splitintegral}
is asymptotic to 
\begin{align}
 & \frac{z_{1}^{c+p+int}h_{p}(z_{2}-z_{1})^{c+int}}{z_{1}^{n}2^{c+int}z_{1}^{c+int}(z_{2}+z_{1})^{c+int}}\int_{(0^{-})}^{\left(\ln^{5}n/n\right)\pm i\epsilon}\frac{(1-e^{z})^{c+p+int}}{e^{nz}}dz\nonumber \\
\sim & \frac{z_{1}^{c+p+int}h_{p}(z_{2}-z_{1})^{c+int}}{z_{1}^{n}2^{c+int}z_{1}^{c+int}(z_{2}+z_{1})^{c+int}}\int_{(0^{-})}^{\left(\ln^{5}n/n \right)\pm i\epsilon}(-z)^{c+p+int}e^{-nz}dz\nonumber \\
= & \frac{z_{1}^{c+p+int}h_{p}(z_{2}-z_{1})^{c+int}}{z_{1}^{n}2^{c+int}z_{1}^{c+int}(z_{2}+z_{1})^{c+int}n^{c+p+1+int}}\int_{(0^{-})}^{\left(\ln^{5}n\right)\pm in\epsilon}(-z)^{c+p+int}e^{-z}dz.\label{eq:mainint}
\end{align}
The following auxiliary lemma helps us further refine this estimate.
\begin{lem}
\label{lem:smalltasympmainterm} Suppose, as in the statement of Proposition \ref{prop:asympsmallt}, that $e^{-\ln^4n}/n \ll t \ll \ln^4 n /n$. If $n\epsilon=o(1)$, then the following asymptotic equivalence holds:
\[
\int_{(0^{-})}^{\ln^{5}n\pm in\epsilon}(-z)^{c+p+int}e^{-z}dz\sim2i\sin\pi(c+p+1+int)\Gamma(c+p+1+int).
\]
\end{lem}

\begin{proof}
We apply the Hankel contour representation (see for example \cite{moretti}) for the Gamma function 
\[
\Gamma(s)=\frac{1}{2i\sin\pi s}\int_{+\infty}^{(0^{-})}e^{-z}(-z)^{s-1}dz \qquad (s \neq 0, -1, -2, \ldots)
\]
to rewrite 
\[
\int_{(0^{-})}^{\ln^{5}n\pm in\epsilon}(-z)^{c+p+int}e^{-z}dz
\]
as 
\begin{equation}
2i\sin\pi(c+p+1+int)\Gamma(c+p+1+int)-\int_{\ln^{5}n+in\epsilon}^{+\infty+in\epsilon}(-z)^{c+p+int}e^{-z}dz+\int_{\ln^{5}n-in\epsilon}^{+\infty-in\epsilon}(-z)^{c+p+int}e^{-z}dz.\label{eq:maintermGamma}
\end{equation}

In the case $e^{-\ln^{4}n}\ll nt=\mathcal{O}(1)$, we have 
\begin{equation}
2i\sin\pi(c+p+1+int)\Gamma(c+p+1+int)\gg e^{-\ln^{4}n}.\label{eq:sinegammaprod}
\end{equation}
Moreover, in this case, the second and the third terms in \eqref{eq:maintermGamma}
are bounded by 
\begin{equation} \label{eq:absbound}
\int_{\ln^{5}n\pm in\epsilon}^{+\infty\pm in\epsilon}|z|^{c+p}e^{-nt\Arg(-z)}e^{-\Re z}d|z|.
\end{equation}
Since $n\epsilon=o(1)$, we have $|z|=\Re z+o(1)$, whence by the
asymptotic behavior of the upper incomplete Gamma function (see \cite{dingle, olver}), the integral in \eqref{eq:absbound} is $\mathcal{O}(e^{-\ln^{5}n}\ln^{5(c+p)}n)$. The result in this case now follows.\\
If, on the other hand, $nt\gg1$, then the Stirling formula 
\[
\Gamma(s)=\exp\left((s-1/2)\Log s-s+\frac{1}{2}\ln2\pi+\mathcal{O}(s^{-1})\right)\qquad(|s|\gg1)
\]
and the fact $nt\ll\ln^{4} n$ imply that 
\begin{align*}
& |\Gamma(c+p+1+int)| \\
 & =\exp\left(\left(c+p+\frac{1}{2}\right)\ln|c+p+1+int|-nt\Arg(c+p+1+int)-(c+p+1)+\frac{1}{2}\ln2\pi+\mathcal{O}\left(\frac{1}{nt}\right)\right)\\
 & \gg\exp(-\pi\ln^{4}n),
\end{align*}
from which we deduce 
\[
\left|2i\sin\pi(c+p+1+int)\Gamma(c+p+1+int)\right|\gg\exp\text{\ensuremath{\left(n\pi t-\pi\ln^{4}n\right)}.}
\]
The claim follows from the fact that 
\[
\int_{\ln^{5}n\pm in\epsilon}^{+\infty\pm in\epsilon}|z|^{c+p}e^{-nt\Arg(-z)}e^{-\Re z}d|z|=\mathcal{O}\left(e^{nt\pi-\ln^{5}n}\ln^{5(c+p)}n\right).
\]
The proof of Lemma \ref{lem:smalltasympmainterm} is complete.
\end{proof}
We continue with the proof of Proposition \ref{prop:asympsmallt} by finding a bound for the second integral in \eqref{eq:splitintegral}:
\[
z_{1}\int_{ln^{5}n/n+i\epsilon}^{\infty+i\epsilon}\psi(z_{1}e^{z})e^{-n\phi(z_{1}e^{z},t)}e^{z}dz.
\]
Recall from equation \eqref{eq:integrand}
that 
\[
\psi(z_{1}e^{z})e^{-n\phi(z_{1}e^{z},t)}=\frac{h(z_{1}e^{z})}{z_{1}^{n+1}e^{(n+1)z}}\frac{(z_{1}-z_{1}e^{z})^{c}(z_{2}-z_{1}e^{z})^{c}}{(z_{1}+z_{1}e^{z})^{c}(z_{2}+z_{1}e^{z})^{c}}\left(\frac{(z_{1}-z_{1}e^{z})(z_{2}-z_{1}e^{z})}{(z_{1}+z_{1}e^{z})(z_{2}+z_{1}e^{z})}\right)^{nit}.
\]
With $z=u+i\epsilon$, $\ln^{5}n/n\le u<\infty$, we have 
\begin{align*}
|z_{2}-z_{1}e^{z}| & \ge|\Im(z_{2}-z_{1}e^{z})|\\
 & =z_{1}e^{u}\sin\epsilon\\
 & \gg\epsilon,
\end{align*}
and consequently, 
\[
(z_{2}-z_{1}e^{z})^{c}=\mathcal{O}\text{\ensuremath{\left(\frac{1}{\epsilon^{|c|}}+e^{u|c|}\right)}. }
\]
We conclude that 
\[
\frac{h(z_{1}e^{z})(z_{2}-z_{1}e^{z})^{c}(z_{1}-z_{1}e^{z})^{c}}{(z_{1}+z_{1}e^{z})^{c}(z_{2}+z_{1}e^{z})^{c}}=\begin{cases}
\mathcal{O}\left(\frac{1}{\epsilon^{B}}+\frac{n^{|c+p|}}{\ln^{5(c+p)}n}\right) & \text{if }z=\mathcal{O}(1)\\
\mathcal{O}(e^{Au}) & \text{if }z\gg1
\end{cases}
\]
for some constants $A$(depending on the degree of $h$) and $B$(depending
on $c$ and the number of poles of $h$ on $[z_{1},\infty)$).

We also note that 
\[
\left|\left(\frac{(z_{1}-z_{1}e^{z})(z_{2}-z_{1}e^{z})}{(z_{1}+z_{1}e^{z})(z_{2}+z_{1}e^{z})}\right)^{nit}\right|=\exp^{nt}\left(\Arg(z_{1}-z_{1}e^{z})+\Arg(z_{2}-z_{1}e^{z})-\Arg(z_{1}+z_{1}e^{z})-\Arg(z_{2}+z_{1}e^{z})\right),
\]
and that for $z=u+i\epsilon$ and $|z| \ll 1$,
\[
\Arg(z_{1}-z_{1}e^{z})+\Arg(z_{2}-z_{1}e^{z})-\Arg(z_{1}+z_{1}e^{z})-\Arg(z_{2}+z_{1}e^{z})=\Arg(-z)+\mathcal{O}(z).
\]
Thus, there exists a small $\delta$ (independent of $n$) such that if
$u<\delta$ then 
\[
\left|\left(\frac{(z_{1}-z_{1}e^{z})(z_{2}-z_{1}e^{z})}{(z_{1}+z_{1}e^{z})(z_{2}+z_{1}e^{z})}\right)^{nit}\right|<e^{n\pi t}.
\]
For other values of $u$, we note that geometrically 
\[
\left|\Arg(z_{1}-z_{1}e^{z})-\Arg(z_{1}+z_{1}e^{z})\right|
\]
is the sum of two angles of the triangle with vertices $0$, $z_{1}$,
and $(z_{1}+z_{1}e^{z})/2$, which is less than $\pi$. The same inequality
holds for 
\[
\left|\Arg(z_{2}-z_{1}e^{z})-\Arg(z_{2}+z_{1}e^{z})\right|.
\]
We conclude that for $u\ge\delta$, 
\[
\left|\left(\frac{(z_{1}-z_{1}e^{z})(z_{2}-z_{1}e^{z})}{(z_{1}+z_{1}e^{z})(z_{2}+z_{1}e^{z})}\right)^{nit}\right|<e^{2n\pi t}.
\]
In order to utilize these estimates, we break the range of integration of 
\[
\int_{\ln^{5}n/n+i\epsilon}^{+\infty+i\epsilon}z_{1}\psi(z_{1}e^{z})e^{-n\phi(z_{1}e^{z},t)}e^{z}dz
\]
into three pieces: (i) $\ln^{5}n/n<\Re z<\delta$, (ii) $\delta\le \Re z$ and $\Re z=\mathcal{O}(1)$
and (iii) $\Re z\gg1$. The integral over the first range is 
\[
\frac{e^{n\pi t}}{z_{1}^{n}}\mathcal{O}\left(\left(\frac{1}{\epsilon^{B}}+\frac{n^{|c|}}{\ln^{5c}n}\right)\int_{\ln^{5}n/n}^{\delta}e^{-nu}du\right)=\frac{e^{n\pi t}}{nz_{1}^{n}}\mathcal{O}\left(\frac{e^{-\ln^{5}n}}{\epsilon^{B}}+\frac{n^{|c|}}{\ln^{5c}n}e^{-\ln^{5}n}\right),
\]
the integral over the second range is 
\[
\frac{e^{2\pi nt}}{z_{1}^{n}}\mathcal{O}\left(\left(\frac{1}{\epsilon^{B}}+\frac{n^{|c|}}{\ln^{5c}n}\right)\frac{e^{-n\delta}}{n}\right),
\]
while the integral over the third range is 
\[
\frac{e^{2\pi nt}}{z_{1}^{n}}\mathcal{O}\left(\int_{C}^{\infty}e^{-(n+1)u+Au}du\right)=\frac{e^{2\pi nt}}{z_{1}^{n}}\mathcal{O}\left(\frac{1}{n}e^{-nC}\right)
\]
for some large constant $C$. We recall that $nt\ll\ln^{4}n$. If we choose $\epsilon$
so that in addition to satisfying the condition $n\epsilon=o(1)$ we also have
\[
\frac{1}{\epsilon^{B}}=\mathcal{O}\left(\exp\left(\frac{\ln^{5}n}{2}\right)\right),
\]
then 
\[
z_{1}\int_{\ln^{5}n/n+i\epsilon}^{+\infty+i\epsilon}\psi(z_{1}e^{z})e^{-n\phi(z_{1}e^{z},t)}e^{z}dz=\frac{e^{\pi nt}}{z_{1}^{n}}\mathcal{O}\left(\exp\left(-\frac{\ln^{5}n}{2}\right)\right).
\]
With a similar argument we obtain 
\[
z_{1}\int_{+\infty-i\epsilon}^{\ln^{5}n/n-i\epsilon}\psi(z_{1}e^{z})e^{-n\phi(z_{1}e^{z},t)}e^{z}dz=\frac{e^{\pi nt}}{z_{1}^{n}}\mathcal{O}\left(\exp\left(-\frac{\ln^{5}n}{2}\right)\right).
\]
From equations \eqref{eq:splitintegral}, \eqref{eq:mainint}, Lemma \ref{lem:smalltasympmainterm},
and the fact that 
\[
2i\sin\pi(c+p+1+int)\Gamma(c+p+1+int)\begin{cases}
\gg e^{-\ln^{4}n} & \text{if }e^{-\ln^{4}n}\ll nt=\mathcal{O}(1)\\
\gg\exp\text{\ensuremath{\left(n\pi t-\frac{\pi\ln^{4}n}{4}\right)}} & \text{if }1\ll nt\ll\ln^{4}n
\end{cases},
\]
we conclude that  
\[
\int_{+\infty}^{(z_{1}^{-})}\psi(z)e^{-n\phi(z,t)}dz\sim\frac{2i\sin\pi(c+p+1+int)\Gamma(c+p+1+int)z_{1}^{c+p+int}h_{p}(z_{2}-z_{1})^{c+int}}{z_{1}^{n}2^{c+int}z_{1}^{c+int}(z_{2}+z_{1})^{c+int}n^{c+p+1+int}}
\]
uniformly on $e^{-\ln^{4}n}\ll nt\ll\ln^{4}n$. The proof of Proposition \ref{prop:asympsmallt} is complete. 
\end{proof}
\begin{rem}
\label{rem:smallt}If $c+p\notin\mathbb{Z}^{+}$, we do not require
the condition $e^{-\ln^{4}n}\ll nt$ for the estimate in equation \eqref{eq:sinegammaprod}.
Thus the asymptotics in Proposition \ref{prop:asympsmallt} hold
for $nt\ll\ln^{4}n$ if $c+p\notin\mathbb{Z}^{+}$. 
\end{rem}

\subsection{The location of the zeros of $\{H_n\}_{n=0}^{\infty}$} \label{subsec:zerolocation}
With the asymptotic analysis complete, in this section we establish that for all $n\gg 1$, the zeros fo $H_n$ lie on the line $\Re s=c$ (c.f. Theorem \ref{thm:maintheorem}) except perhaps a finite number of real zeros, whose asymptotic locations we also identify. We begin with a technical, but crucial result.
\begin{lem}
\label{lem:changeargsmallt} Suppose $\tau_{1}$ and $\tau_{2}$ are
constant multiples of $e^{-\ln^{4}n}/n$ and $\ln^{4}n/n$ and $\alpha$
is the unique angle such that $-\pi<\alpha\le\pi$ and $(c+p+1/2)\pi=2k\pi+\alpha$
for $k\in\mathbb{Z}$. Let $g$ and $p$ be defined as in equation \eqref{eq:gzetadef}. Then 
\begin{itemize}
\item[(i)] $p(\zeta(t))\ne0$ for $\tau_{1}\le t\le\tau_{2}$, and 
\item[(ii)] $\displaystyle{
\Delta\arg_{\tau_{1}\le t\le\tau_{2}}p(\zeta(t))=\frac{1}{2}\lim_{\xi\rightarrow0}\Delta\arg_{\xi\le t\le\tau_{2}}g(\zeta(t))+\frac{|c+p|\pi}{2}+\eta}$,
\end{itemize}
where 
\begin{equation}
\eta=\begin{cases}
-\pi/2 & \text{if }c+p<0\\
0 & \text{ if }c+p\ge0\text{ and }\alpha=\pm\frac{\pi}{2}\\
-\alpha & \text{ if }c+p\ge0\text{ and }-\pi/2<\alpha<\pi/2\\
-\alpha-\pi & \text{ if }c+p\ge0\text{ and }-\pi<\alpha<-\pi/2\\
-\alpha+\pi & \text{ if }c+p\ge0\text{ and }\pi/2<\alpha\le\pi.
\end{cases}.\label{eq:etadef}
\end{equation}
\end{lem}

\begin{proof}
The fact that $p(\zeta(t))\ne0$ for $\tau_{1}\le t\le\tau_{2}$ follows
immediately from Proposition \pageref{prop:asympsmallt}.
To establish the claim regarding the change of arguments, we start by recalling that
\[
g(\zeta)=\frac{2\pi\psi^{2}(\zeta)e^{-2n\phi(\zeta)}}{n\phi_{z^{2}}(\zeta,t)}=\frac{2\pi}{n\phi_{z^{2}}(\zeta,t)}\cdot\frac{h^{2}(\zeta)}{z^{2n+2}}\frac{(z_{1}-z)^{2c}(z_{2}-z)^{2c}}{(z_{1}+z)^{2c}(z_{2}+z)^{2c}}\left(\frac{(z_{1}-z)(z_{2}-z)}{(z_{1}+z)(z_{2}+z)}\right)^{2nit}.
\]
Since $z_{1}-z=-iz_{1}t+\mathcal{O}(t^{2})$, on $\xi\le t\le\tau_{2}$ the change in the argument of 
\[
h^{2}(\zeta)\frac{(z_{1}-z)^{2c}(z_{2}-z)^{2c}}{(z_{1}+z)^{2c}(z_{2}+z)^{2c}}=(z_1-z)^{2 p^*}H(t) \sim (-z_1^2t^2)^{p^*}+\mathcal{O}(t^3)
\]
 is $o(1)$. Thus, using the same
computations in the proof of Lemma 39 in \cite{cfkt}, we conclude that 
\begin{equation} \label{eq:argest}
\lim_{\xi\rightarrow0}\Delta\arg_{\xi\le t\le\tau_{2}}g(\zeta(t))=2n\tau_{2}\ln\frac{\tau_{2}(z_{2}-z_{1})}{2(z_{1}+z_{2})}-2n\tau_{2}+\frac{\pi}{2}+o(1).
\end{equation}

 For $\tau_{1}\le t\le\tau_{2}$, the change
of the arguments of the factors $2^{c+int}$, $(z_{2}-z_{1})^{c+int}$,
$(z_{2}+z_{1})^{c+int}$, and $e^{(c+p+1+int)\ln n}$ in the expression
\[
\frac{2i\pi h_{p}(z_{2}-z_{1})^{c+int}}{z_{1}^{n-p}2^{c+int}(z_{2}+z_{1})^{c+int}n^{c+p+1+int}\Gamma(-c-p-int)}
\]
 are $n(\tau_{2}-\tau_{1})\ln2$, $n(\tau_{2}-\tau_{1})\ln(z_{2}-z_{1})$,
$n(\tau_{2}-\tau_{1})\ln(z_{2}+z_{1})$, and $n(\tau_{2}-\tau_{1})\ln n$
respectively.

We next compute the change in argument of the factor $\Gamma(-c-p-int)$
, $\tau_{1}\le t\le\tau_{2}$, which is given by the expression 
\[
\left.\Im\Log\Gamma(-c-p-int)\right|_{\tau_{1}}^{\tau_{2}},
\]
where the function $\Log\Gamma(s)$ is defined as 
\[
\Log\Gamma(s)=-\gamma s-\Log s+\sum_{k=1}^{\infty}\left(\frac{s}{k}-\Log(1+s/k)\right).
\]
Using the Stirling formula, 
\[
\Log\Gamma(s)\sim(s-1/2)\Log s-s+1/2\Log(2\pi)+\mathcal{O}(1/s),\qquad\text{for }|s|\rightarrow\infty\text{ and }|\Arg s|\le\pi-\delta,
\]
we conclude that 
\begin{align*}
 & \Im\Log\Gamma(-c-p-in\tau_{2})\\
= & \Im\left((-c-p-1/2-in\tau_{2})\Log(-c-p-in\tau_{2}\right)+n\tau_{2}+\mathcal{O}(1/\ln^{4}n)\\
= & -n\tau_{2}\ln|c+p+in\tau_{2}|+(c+p+1/2)\frac{\pi}{2}+n\tau_{2}+\mathcal{O}(1/\ln^{4}n).
\end{align*}
Employing the estimate
\[
\ln|c+p+in\tau_{2}|=\ln \left| in \tau_2\left(1+\frac{c+p}{i n \tau_2}\right) \right|=\ln(n\tau_{2})+\mathcal{O}\left(\frac{1}{n^{2}\tau_{2}^{2}}\right),
\]
the last expression becomes 
\[
-n\tau_{2}\ln(n\tau_{2})+(c+p+1/2)\frac{\pi}{2}+n\tau_{2}+\mathcal{O}\left(\frac{1}{\ln^{4}n}\right).
\]
If $-c-p>0$, then the fact that $n\tau_{1}\asymp e^{-\ln^{4}n}$
implies that 
\[
\Im\Log\Gamma(-c-p-in\tau_{1})=\mathcal{O}\left(e^{-\ln^{4}n}\right),
\]
and consequently 
\[
\Delta\arg_{\tau_{1}\le t\le\tau_{2}}\Gamma(-c-p-int)=-n\tau_{2}\ln(n\tau_{2})+(c+p+1/2)\frac{\pi}{2}+n\tau_{2}+\mathcal{O}\left(\frac{1}{\ln^{4}n}\right).
\]
If, on the other hand, if $c+p\ge0$, then the identity
\[
\Gamma(-c-p-int)=\frac{\pi}{\sin\pi(c+p+1+int)\Gamma(c+p+1+int)}
\]
implies that 
\[
\Delta\arg_{\tau_{1}\le t\le\tau_{2}}\Gamma(-c-p-int)=-\Delta\arg_{\tau_{1}\le t\le\tau_{2}}\sin\pi(c+p+1+int)-\Delta\arg_{\tau_{1}\le t\le t_{2}}\Gamma(c+p+1+int).
\]
Using the conjugate of the gamma function, we write
\[
-\Delta\arg_{\tau_{1}\le t\le t_{2}}\Gamma(c+p+1+int)=\Delta\arg_{\tau_{1}\le t\le\tau_{2}}\Gamma(c+p+1-int)=-n\tau_{2}\ln(n\tau_{2})-(c+p+1/2)\frac{\pi}{2}+n\tau_{2}+\mathcal{O}\left(\frac{1}{\ln^{4}n}\right).
\]
Analyzing the change in the argument of $\sin \pi(c+p+int)$ requires further considerations. To this end, recall that 
\[
\sin\pi(c+p+1+int)=\frac{e^{-\pi nt+i\pi(c+p+1)}-e^{\pi nt-i\pi(c+p+1)}}{2i}=\frac{-e^{\pi nt-i\pi(c+p+1)}}{2i}\left(e^{-2 \pi nt} +1\right),
\]
and whence
\[
\Im\sin\pi(c+p+1+int)=\frac{1}{2}\left(e^{-\pi nt}-e^{\pi nt}\right)\cos\pi(c+p+1).
\]
It is immediate that if $c+p+1/2\in\mathbb{Z}$, then 
\[
\Delta\arg_{\tau_{1}\le t\le\tau_{2}}\sin\pi(c+p+1+int)=0.
\]
On the other hand, if $c+p+1/2\notin\mathbb{Z}$, then $\sin\pi(c+p+1+int)\notin\mathbb{R}$,
and
\[
\Delta\arg_{\tau_{1}\le t\le\tau_{2}}\sin\pi(c+p+1+int)=\Arg\sin\pi(c+p+1+in\tau_{2})-\Arg\sin\pi(c+p+1+in\tau_{1}).
\]
We write
\begin{align*}
\Arg\sin\pi(c+p+1+in\tau_{2}) & =\Arg\left(\frac{-e^{\pi n\tau_{2}-i\pi(c+p+1)}}{2i}\right)+\mathcal{O}\left(e^{-2\pi n\tau_{2}}\right)\\
 & =-\alpha+\mathcal{O}(e^{-2\pi\ln^{4}n}),
\end{align*}
where $\alpha$ is the unique angle such that $-\pi<\alpha\le\pi$
and $(c+p+1/2)\pi=2k\pi+\alpha$ for $k\in\mathbb{Z}$. Note that
$\alpha$ is given explicitly by the formula 
\[
\alpha=((c+p+3/2)\pi\mod2\pi)-\pi.
\]
If $c+p\in\mathbb{Z}$, then the Taylor expansion of the sine function
yields
\[
\Arg\sin\pi(c+p+1+in\tau_{1})=(-1)^{c+p+1}\frac{\pi}{2}+\mathcal{O}\left(e^{-\ln^{4}n}\right).
\]
If $c+p\notin\mathbb{Z}$, then 
\[
\Arg\sin\pi(c+p+1+in\tau_{1})=\begin{cases}
\mathcal{O}\left(e^{-\ln^{4}n}\right) & \text{ if }\sin\pi(c+p+1)>0\\
\pi & \text{ if }\sin\pi(c+p+1)<0\text{ and }\cos\pi(c+p+1)>0\\
-\pi & \text{ if }\sin\pi(c+p+1)<0\text{ and }\cos\pi(c+p+1)<0.
\end{cases}
\]
Combining these cases we conclude for any $c,p\in\mathbb{R}$,
\[
\Delta\arg_{\tau_{1}\le t\le\tau_{2}}\Gamma(-c-p-int)=-n\tau_{2}\ln(n\tau_{2})+n\tau_{2}-\frac{|c+p|\pi}{2}-\frac{\pi}{4}-\eta+\mathcal{O}\left(\frac{1}{\ln^{4}n}\right),
\]
and finally,
\[
\Delta\arg_{\tau_{1}<t\le\tau_{2}}p(\zeta(t))=n\tau_{2}\ln\frac{(z_{2}-z_{1})\tau_{2}}{2(z_{2}+z_{1})}-n\tau_{2}+\frac{|c+p|\pi}{2}+\frac{\pi}{4}+\eta+\mathcal{O}\left(\frac{1}{\ln^{4}n}\right).
\]
Given equation \eqref{eq:argest}, the result now follows. 
\end{proof}
We next identify a suitable curve on which we will compute the change of argument of $g$. Let $\gamma$ be the simple closed curve with counter clockwise orientation
formed by the traces of $\zeta(t),\overline{\zeta(t)},-\zeta(t),$
and $-\overline{\zeta(t)}$ for $0\le t\le T$ and small deformations
around 
\begin{equation}
\begin{cases}
\pm i\sqrt{z_{1}z_{2}},\pm\zeta(T_{1}),\pm\overline{\zeta(T_{1})} & \text{ if }z_{1}^{2}-6z_{1}z_{2}+z_{2}^{2}<0\\
\pm i\zeta(T) & \text{ if }z_{1}^{2}-6z_{1}z_{2}+z_{2}^{2}\ge0
\end{cases}\label{eq:zerosphi''}
\end{equation}
such that the region enclosed by $\gamma$ contains the points defined
in \eqref{eq:zerosphi''}. We also deform $\gamma$ around $\pm z_{1}$
so that the cuts $(-\infty,-z_{1}]$ and $[z_{1},\infty)$ lie outside
this region (see Figure \ref{fig:gammacurve}). Using the residue theorem
(for detailed computations, see \cite{cfkt} equation (2.86)), we
find that
\[
\frac{1}{2\pi i}\int_{\gamma}\frac{g'(\zeta)}{g(\zeta)}d\zeta=\begin{cases}
-2n-6 & \text{ if }z_{1}^{2}-6z_{1}z_{2}+z_{2}^{2}<0\\
-2n-2 & \text{ if }z_{1}^{2}-6z_{1}z_{2}+z_{2}^{2}\ge0
\end{cases},
\]
since the values of $c$ and $p$ do not affect the integral.

\begin{figure}
\begin{centering}
\includegraphics[scale=0.3]{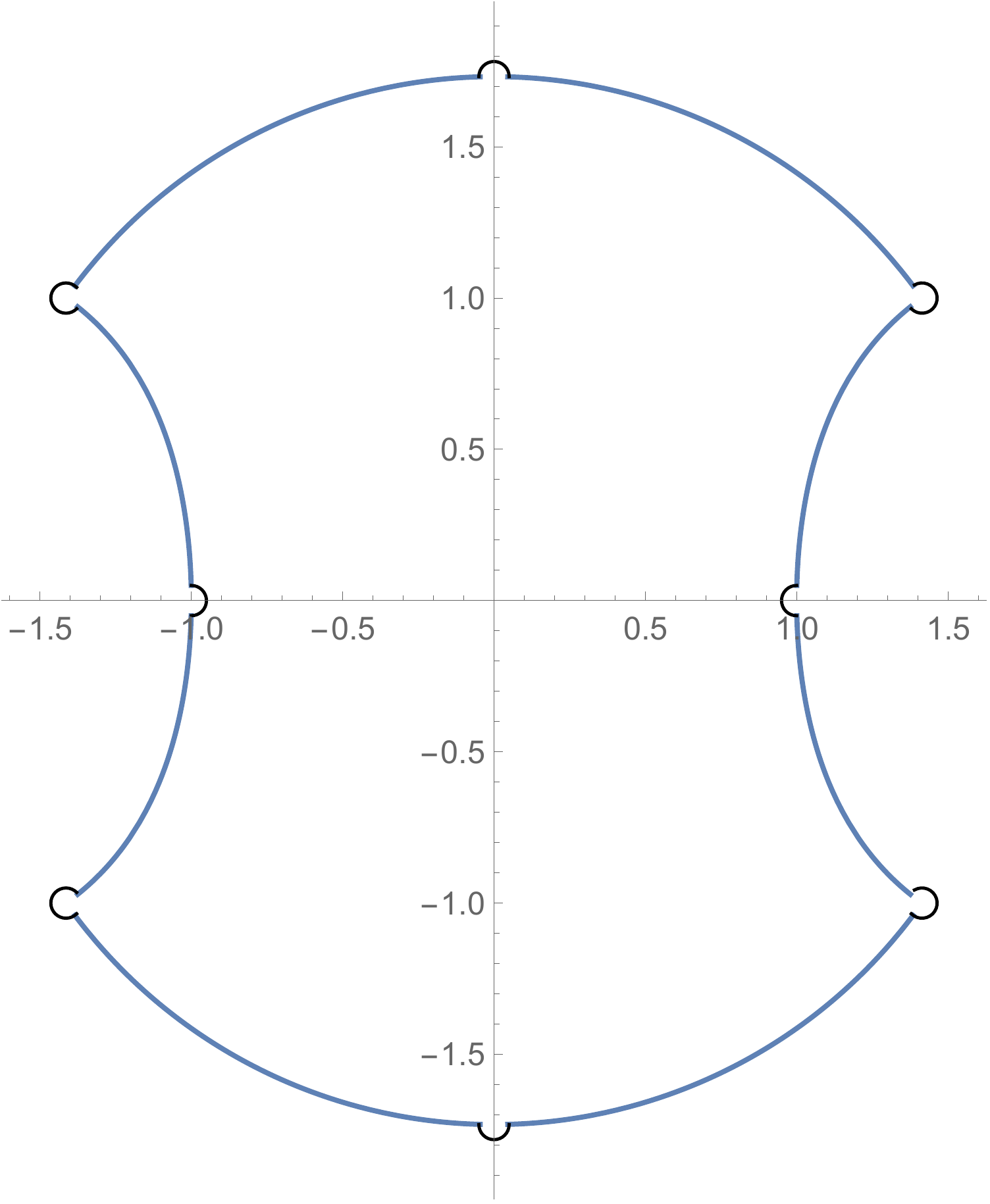} $\qquad$\includegraphics[scale=0.3]{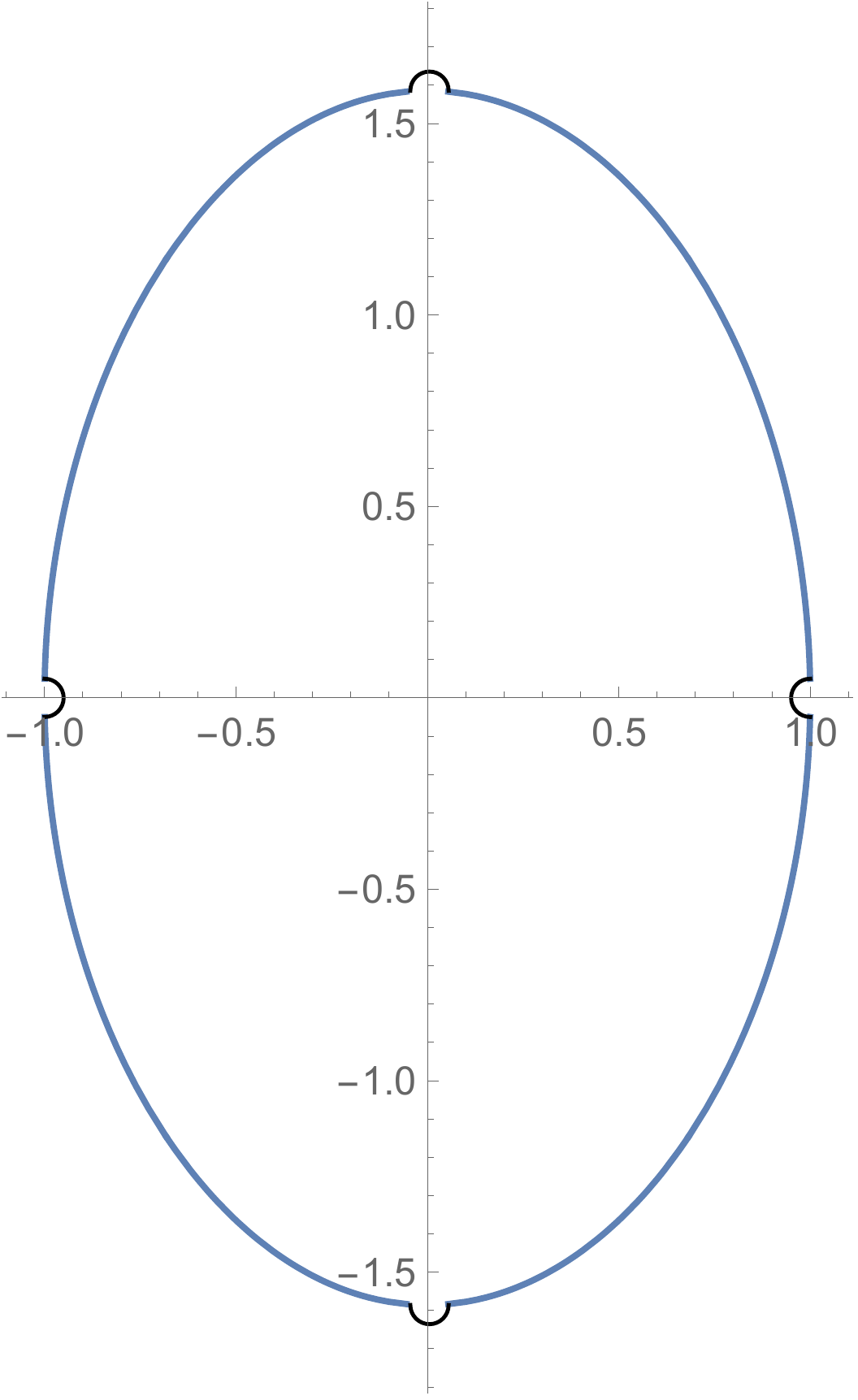} 
\par\end{centering}
\caption{\label{fig:gammacurve}The curve $\gamma$ for $(z_{1},z_{2})=(1,3)$
(left) and $(1,7)$ (right)}
\end{figure}

Let $\gamma_{1}$ be the portion of $\gamma$ in the first quadrant.
Exploiting the symmetry $g(\overline{\zeta})=\overline{g(\zeta)}$ and
$g(-\zeta)=g(\zeta)$, we conclude that 
\[
\Delta_{\gamma_{1}}\arg g(\zeta)=\frac{\Delta_{\gamma}\arg g(\zeta)}{4}=\begin{cases}
-(n+3)\pi & \text{ if }z_{1}^{2}-6z_{1}z_{2}+z_{2}^{2}<0\\
-(n+1)\pi & \text{ if }z_{1}^{2}-6z_{1}z_{2}+z_{2}^{2}\ge0
\end{cases}.
\]

\begin{rem}
\label{rem:changearg}In the case $t\rightarrow T$ and $t\rightarrow T_{1}$
(when $T\ne T_{1})$ , the values $c$ and $p$ only affect the change
in argument of $g(\zeta(t))$ by $o(1)$. Thus the following results
follow directly from \cite{cfkt}. 
\end{rem}

\begin{enumerate}
\item If $\tau<T$ such that $T-\tau\ll1/n^{2/3}$, then by Lemma 40 in \cite{cfkt},
\[
\lim_{\xi\rightarrow0}\Delta\arg_{\tau\le t\le T-\xi}g(\zeta(t))\ll1.
\]
\item If $T=T_{2}$ and $|T_{1}-\tau|\ll1/n$, then by Lemmas 41 and 42 in \cite{cfkt},
\begin{align*}
\lim_{\xi\to0}\Delta_{T_{1}+\xi<t<\tau}\arg g(\zeta(t)) & \ll1, \quad \textrm{and}\\
\lim_{\xi\rightarrow0}\Delta_{\tau\le t<T_{1}-\xi}\arg g(\zeta(t)) & \ll1.
\end{align*}
\item Lemma 37 in \cite{cfkt} shows that there exists some $|C|<\pi/2+o(1)$, such that
\[
\Delta_{\ln^{2}n/n^{2/3}<T-t\ll1/n^{2/3}}\arg p(\zeta(t)=\frac{1}{2}\Delta_{\ln^{2}n/n^{2/3}<T-t\ll1/n^{2/3}}\arg g(\zeta(t))+C.
\]
\item If $T=T_{2}$ and $p(\zeta(T_{1}))\ne0$, then $p(\zeta(t))\ne0$
on $(T_{1}-\ln^{2}n/n^{2/3},T_{1}+\ln^{2}n/n^{2/3})$ and by Lemma 38 in \cite{cfkt}, there exists some $|C|<\pi+o(1)$ such that
\begin{align*}
 & \Delta_{T_{1}-\ln^{2}n/n^{2/3}<t<T_{1}+\ln^{2}n/n^{2/3}}\arg p(\zeta(t))\\
= & \frac{1}{2}\Delta_{\ln^{2}n/n^{2/3}<T_{1}-t\ll1/n}\arg g(\zeta)+\frac{1}{2}\Delta_{1/n\ll t-T_{1}<\ln^{2}n/n^{2/3}}\arg g(\zeta)+C.
\end{align*}
\item An argument completely analogous to that on p.55 in \cite{cfkt} shows that the change in the argument of $g(\zeta(t))$ on the small arcs of $\gamma_{1}$
around $\zeta(T)$ and $\zeta(T_{1})$ (when $T\ne T_{1}$) are $-\pi/2$
and $-3\pi/2$ respectively . 
\end{enumerate}
Finally, as $\zeta\rightarrow z_{1}$, 
\[
g(\zeta)=\frac{2\pi\psi^{2}(\zeta)e^{-2n\phi(\zeta)}}{n\phi_{z^{2}}(\zeta)}\sim c_{1}(z_{1}-\zeta)^{2c+2p+1},\qquad(c_{1}\in\mathbb{C} \setminus \{0\}),
\]
and 
\[
\exp\left(2n(\zeta-z_{1})(z_{1}-z_{2})\Log(z_{1}-\zeta)/z_{1}\right)\rightarrow1.
\]
Thus the change in argument of $g(\zeta)$ on the small arc of $\gamma_{1}$
around $z_{1}$ is $-(c+p+1/2)\pi+o(1)$. \\
In contrast to the polynomials discussed in \cite{cfkt} -- which have at most two real zeros -- the family of polynomials we are treating in the current paper can have several real zeros, whose location changes with $n$. The next result gives a lower bound on the number of real zeros of $H_n$ if $c+p>0$, and also describes the asymptotic behavior of these zeros as $n \to \infty$.

\begin{lem}
\label{lem:realzeros} Let $\{H_n\}_{n=0}^{\infty}$ be as in the statement of Theorem \ref{thm:maintheorem}. If $c+p>0$, then $H_{n}(s)$ has at least $2\left\lceil c+p\right\rceil $ many real zeros
which approach $c\pm(c+p+1-k)$, $0<k\le\left\lceil c+p\right\rceil $ as $n \to \infty$.
\end{lem}

\begin{proof}
For $x\in\mathbb{R}$, the Cauchy integral formula yields 
\begin{align}
H_{n}(c+x) & =\frac{n!}{2\pi i}\ointctrclockwise_{|z|=\epsilon}\frac{h(z)}{z^{n+1}}\frac{(z_{1}-z)^{c}(z_{2}-z)^{c}}{(z_{1}+z)^{c}(z_{2}+z)^{c}}\left(\frac{(z_{1}-z)(z_{2}-z)}{(z_{1}+z)(z_{2}+z)}\right)^{x}dz\label{eq:Cauchyreal}\\
 & =\frac{n!}{2\pi i}\int_{\Gamma_{1}\cup\Gamma_{2}}\frac{h(z)}{z^{n+1}}\frac{(z_{1}-z)^{c}(z_{2}-z)^{c}}{(z_{1}+z)^{c}(z_{2}+z)^{c}}\left(\frac{(z_{1}-z)(z_{2}-z)}{(z_{1}+z)(z_{2}+z)}\right)^{x}dz,\nonumber 
\end{align}
where $\Gamma_{1}$ and $\Gamma_{2}$ are two loops around two cuts
$(-\infty,-z_{1}]$ and $[z_{1},\infty)$ oriented counter clockwise. Using the substitution $z\mapsto-z$ and the fact
that 
\[
h(z)\frac{(z_{1}-z)^{c}(z_{2}-z)^{c}}{(z_{1}+z)^{c}(z_{2}+z)^{c}}
\]
is an even function, we see that the integral over $\Gamma_{1}$ is equal to
\[
(-1)^{n}\int_{\Gamma_{2}}\frac{h(z)}{z^{n+1}}\frac{(z_{1}-z)^{c-x}(z_{2}-z)^{c-x}}{(z_{1}+z)^{c-x}(z_{2}+z)^{c-x}}dz.
\]
We apply Remark \ref{rem:smallt} to $t=0$ and $c$ replaced by $c+x$
to conclude that if $c+p+x\notin\mathbb{Z}^{+}$, then the integral
over $\Gamma_{2}$ is asymptotic to 
\[
\frac{2ih_{p}(z_{2}-z_{1})^{c+x}\sin\pi(c+p+x+1)\Gamma(c+p+x+1)}{z_{1}^{n-p}2^{c+x}(z_{2}+z_{1})^{c+x}n^{c+p+x+1}}.
\]
With the same application to the case when $c$ replaced by $c-x$, we
conclude that 
\begin{align}
H_{n}(c+x) & \sim\frac{n!h_{p}(z_{2}-z_{1})^{c+x}\sin\pi(c+p+x+1)\Gamma(c+p+x+1)}{\pi z_{1}^{n-p}2^{c+x}(z_{2}+z_{1})^{c+x}n^{c+p+x+1}}\nonumber \\
 & +(-1)^{n}\frac{n!h_{p}(z_{2}-z_{1})^{c-x}\sin\pi(c+p-x+1)\Gamma(c+p-x+1)}{\pi z_{1}^{n-p}2^{c-x}(z_{2}+z_{1})^{c-x}n^{c+p-x+1}}\label{eq:hnc+x}
\end{align}
if $c+p\pm x\notin\mathbb{Z}^{+}$ and $x\ne0$. For any small fixed
$\delta>0$ (independent of $n$), we consider the intervals
\begin{equation}
J_{k}=[c+p+1-k-\delta,c+p+1-k+\delta],\qquad0<k<c+p+1-\delta.\label{eq:J_kint}
\end{equation}
For each $k$, the values of $\sin(c+p-x+1)$ when $x$
is at the endpoints of $J_{k}$ are $(-1)^{k-1}\sin\delta$ and $(-1)^{k}\sin\delta$.
Also at these endpoints $c+p\pm x\notin\mathbb{Z}^{+}$ (for small
$\delta)$, $\Gamma(c+p-x+1)>0$, and the second term of \eqref{eq:hnc+x}
dominates the first term when $n$ is large. Thus, by the Intermediate
Value Theorem, each interval $J_{k}$ contain at least one zero of $H_{n}(c+x)$.
We deduce that $H_{n}(c+x)$ has at least $\left\lceil c+p\right\rceil $positive
real zeros. The substitutions $z$ by $-z$ and $x$ by $-x$ in equation \eqref{eq:Cauchyreal} yield
\[
H_{n}(c-x)=(-1)^{n}H_{n}(c+x).
\]
The result now follows from the fact that if $x$ is a real zero
of $H_{n}(c+x)$, then so is $-x$. 
\end{proof}
We now turn our attention to the proof of Theorem \ref{thm:maintheorem}.
In addition to the number of real zeros of $H_{n}(s)$ given in Lemma \ref{lem:realzeros},
we will count the number of zeros of $H_{n}(c+nit)$ on $t\in(0,T)$
and compare this number with the degree of $H_{n}(s)$. We start with
a lemma concerning the degree of $H_n(s)$. 
\begin{lem}
\label{lem:degree} Let $\{H_n\}_{n=0}^{\infty}$ be defined as in Theorem \ref{thm:maintheorem}. Then for each $n \geq 0$, polynomial $H_{n}(c+x)$ has degree $n$ and
the sign of its leading coefficient is $(-1)^{n}$. 
\end{lem}

\begin{proof}
Since 
\[
H_{n}(c-x)=(-1)^{n}H_{n}(c+x),
\]
it suffices to prove that $H_{n}(c-x)$ has degree $n$, and that its leading
coefficient is positive. The generating function for $H_{n}(c-x)$ is given by
is 
\[
h(z)\frac{(z_{1}-z)^{c}(z_{2}-z)^{c}}{(z_{1}+z)^{c}(z_{2}+z)^{c}}(1-z/z_{1})^{-x}(1-z/z_{2})^{-x}(1+z/z_{1})^{x}(1+z/z_{2})^{x}.
\]
For each $k \in \mathbb{N}$, the coefficient of $z^k$ in the binomial expansion
of each factor $(1-z/z_{1})^{-x}$, $(1-z/z_{2})^{-x}$, $(1+z/z_{1})^{x}$,
and $(1+z/z_{2})^{x}$ is a polynomial of degree $k$ in $x$ with
a positive leading coefficient. Thus, given an $n \in \mathbb{N}$, the coefficient of $z^n$ in the product 
\[
(1-z/z_{1})^{-x}(1-z/z_{2})^{-x}(1+z/z_{1})^{x}(1+z/z_{2})^{x}
\]
is of the form
\[
\sum_{i+j+k+\ell=n} p_i(x)p_j(x)p_k(x)p_{\ell}(x),
\]
where each factor of each summand -- and hence the entire expression -- has a positive leading coefficient, and degree equal to its index.
We expand 
\[
h(z)\frac{(z_{1}-z)^{c}(z_{2}-z)^{c}}{(z_{1}+z)^{c}(z_{2}+z)^{c}}
\]
as a power series in $z$ (with constant coefficients) and deduce
that $H_{n}(c-x)$ has degree $n$, and the sign of its leading coefficient
is the same as the sign of the constant coefficient of this series which
is $h(0)>0$. 
\end{proof}
The final piece in accounting for all of the zeros of $H_n(s)$ is provided by the fact (to be proven in short order) that the total number of real zeros of $H_{n}(s)$ and
those on $c+it$ (except the possible zero at $c$) is at least 
\begin{equation}
\begin{cases}
n-2 & \text{ if }2\mid n\\
n-3 & \text{ if }2\nmid n
\end{cases}.\label{eq:lowerboundimzeros}
\end{equation}
 Assuming this fact, we now provide an argument to complete the proof
of Theorem \ref{thm:maintheorem}. If $n$ is odd, then
we let $x=0$ in $H_{n}(c-x)=(-1)^{n}H_{n}(c+x)$ to conclude that
$c$ is a zero of $H_{n}(s)$. It thus remains to account for the two possible missing zeros of $H_n(s)$ regardless of the parity of $n$. Since the degree of $H_{n}(s)$ is
$n$, and the zeros of $H_{n}(s)$ are symmetric about the real line
and the line $c+it$, it suffices to show that the possible two remaining
zeros of $H_{n}(s)$ are not real. Note that $H_{n}(c+x)$ has opposite
signs at the endpoints of each $J_{k}$ (as defined in \eqref{eq:J_kint}). Hence,
 $H_{n}(c+x)$ must have exactly one zero on each $J_{k}$ and consequently
the two remaining zeros cannot lie on any $J_{k}$. Since on the set
\[
(0,c+p+\delta)\backslash\bigcup_{0<k<c+p+1-\delta}J_{k}
\]
the second term of expression \eqref{eq:hnc+x} dominates the first, $H_{n}(c+x)$
does not have zero there. Moreover, it follows from $h_{p}>0$ and the asymptotic expression in  \eqref{eq:hnc+x} that the sign $H_{n}(c+x)$ at $x=c+p+\delta$ is $(-1)^{n}$.
By Lemma \ref{lem:degree}, this is the same as the sign of $\lim_{x\rightarrow\infty}H_{n}(c+x)$, and we conclude that $H_{n}(c+x)$ has no zero on $[c+p+\delta,\infty)$. It follows that the remaining two possible zeros must lie on the line $\Re z=c$, completing the proof of Theorem \ref{thm:maintheorem}. 
\begin{rem}
\label{rem:signat0}In the case $c+p\in\mathbb{Z}^{+}$, \eqref{eq:hnc+x}
implies that $H_{n}(c+x)$ is nonzero on $(0,1-\delta)$ and its sign
is $(-1)^{n+c+p}$ there.
\end{rem}
 We now present the proof of the zero count of $H_n(s)$ claimed in expression \eqref{eq:lowerboundimzeros} above. Since a lower bound for the number of real zeros of $H_{n}(s)$ is provided by Lemma \ref{lem:realzeros}, it remains to count the number of zeros of $H_{n}(c+int)$ on $|t|\in(0,T)$. We recall that for
$t\in(0,T)$, $\pi H_{n}(c+int)$ is the imaginary part of $-i$ times
the real part of $p(\zeta(t))$. It therefore suffices to compute the change in the argument of $p(\zeta(t))$ in order to get a lower bound on the zero count of $H_n(s)$ on the line $\Re z=c$. We proceed by case analysis, depending on whether $T=T_2$ or $T=T_1$ (c.f. equation \eqref{eq:Tdefn}).

\subsection*{Case $T=T_{2}$}

If $T=T_{2}$ and $p(\zeta(T_{1}))\ne0$, then for some $|C|<3\pi/2+o(1)$
and $c_{2}\in\mathbb{R}^{+}$ 
\begin{align}
\Delta_{e^{-ln^{4}n}/n\ll t<T-c_{2}/n^{2/3}}\arg p(\zeta(t)) & =\frac{1}{2}\Delta\arg_{e^{-ln^{4}n}/n\ll t<T-c_{2}/n^{2/3}}g(\zeta(t))+\frac{|c+p|\pi}{2}+\eta+C\nonumber \\
 & =\frac{1}{2}\left(\Delta_{\gamma_{1}}\arg g(\zeta)+(c+p+5/2)\pi\right)+\frac{|c+p|\pi}{2}+\eta+C\nonumber \\
 & =-\frac{n\pi}{2}+\frac{c+p+|c+p|}{2}\pi-\frac{\pi}{4}+\eta+C,\label{eq:changearg}
\end{align}
where $\eta$ is defined as in equation \eqref{eq:etadef} in Lemma \ref{lem:changeargsmallt}.
In the case $c+p<0$, the equation above implies that the number of
zeros of $H_{n}(c+int)$ on $(0,T)$ is at least 
\[
\left\lfloor \frac{n}{2}+\frac{3}{4}-\frac{C}{\pi}\right\rfloor .
\]
It follows from $|C|<3\pi/2+o(1)$ that $H_{n}(s)$ has at least 
\[
\begin{cases}
n-3 & \text{ if }2\nmid n\\
n-2 & \text{ if }2\mid n
\end{cases}
\]
nonreal zeros on the line $\Re s=c+it$.

On the other hand, if $c+p\ge0$, then the number of zeros of $H_{n}(c+int)$
on $(0,T)$ is at least 
\[
\left\lfloor \frac{n}{2}-(c+p)+\frac{\pi}{4}-\frac{\eta+C}{\pi}\right\rfloor .
\]
We conclude from Lemma \pageref{lem:realzeros} that the total number
of real zeros and those on $c+it$ (except the possible zero at $c$)
is at least 
\begin{equation}
2\left\lfloor \frac{n}{2}-(c+p)+\frac{1}{4}-\frac{\eta+C}{\pi}\right\rfloor +2\left\lceil c+p\right\rceil, \label{eq:lowerboundzeros-1}
\end{equation}
where $c+p=2k+\alpha/\pi-1/2$.

If $-\pi<\alpha<-\pi/2$, then $\eta=-\alpha-\pi$. Consequently,
\[
-(c+p)-\frac{\eta+C}{\pi}+\frac{1}{4}=-2k+\frac{7}{4}-\frac{C}{\pi},
\]
and the expression \eqref{eq:lowerboundzeros-1} is at least 
\[
\begin{cases}
2\left(\left\lfloor \frac{n}{2}\right\rfloor -2k\right)+2(2k-1)=n-2 & \text{ if }2 \mid n\\
2\left(\left\lfloor \frac{n}{2}\right\rfloor -2k\right)+2(2k-1)=n-3 & \text{ if }2\nmid n
\end{cases}.
\]
If $\pi/2<\alpha\le\pi$, then $\eta=-\alpha+\pi$ and 
\[
-(c+p)-\frac{\eta+C}{\pi}+\frac{1}{4}=-2k-\frac{1}{4}-\frac{C}{\pi},
\]
from which we see that the expression in \eqref{eq:lowerboundzeros-1} is at least 
\[
\begin{cases}
2\left(\left\lfloor \frac{n}{2}\right\rfloor -2k-2\right)+2(2k+1)=n-2 & \text{ if }2\mid n\\
2\left(\left\lfloor \frac{n}{2}\right\rfloor -2k-2\right)+2(2k+1)=n-3 & \text{ if }2\nmid n
\end{cases}.
\]
If $-\pi/2<\alpha<\pi/2$, then $\eta=-\alpha$ and 
\[
-(c+p)-\frac{\eta+C}{\pi}+\frac{1}{4}=-2k+\frac{3}{4}-\frac{C}{\pi}.
\]
Computing the expression in \eqref{eq:lowerboundzeros-1} again we find that it is at least 
\[
\begin{cases}
2\left(\left\lfloor \frac{n}{2}\right\rfloor -2k-1\right)+4k=n-2 & \text{ if }2\mid n\\
2\left(\left\lfloor \frac{n}{2}\right\rfloor -2k-1\right)+4k=n-3 & \text{ if }2\nmid n
\end{cases}.
\]

If $\alpha=\pi/2$, then $\eta=0$ and 
\[
\Delta_{e^{-ln^{4}n}/n\ll t<T-c_{2}/n^{2/3}}\arg p(\zeta(t))=-\frac{n\pi}{2}+(c+p)\pi-\frac{\pi}{4}+C,
\]
and the expression in  \eqref{eq:lowerboundzeros-1} computes to be at least 
\[
\begin{cases}
2\left(\left\lfloor \frac{n}{2}\right\rfloor -2k-2\right)+4k=n-4 & \text{ if }2\mid n\\
2\left(\left\lfloor \frac{n}{2}\right\rfloor -2k-1\right)+4k=n-3 & \text{ if }2\nmid n.
\end{cases}
\]
%%%%%%%%%%%%%%%%
If $\alpha=-\pi/2$, then $\eta=0$ and 
\[
\Delta_{e^{-ln^{4}n}/n\ll t<T-c_{2}/n^{2/3}}\arg p(\zeta(t))=-\frac{n\pi}{2}+(2k-1)\pi-\frac{\pi}{4}+C.
\]
In this case we find that the expression in \eqref{eq:lowerboundzeros-1} is at least 
\[
\begin{cases}
2\left(\left\lfloor \frac{n}{2}\right\rfloor -2k-1\right)+2(2k-1)=n-4 & \text{ if }2 \mid n\\
2\left(\left\lfloor \frac{n}{2}\right\rfloor -2k\right)+2(2k-1)=n-3 & \text{ if } 2\nmid n.
\end{cases}
\]
The reader will note that if $\alpha=\pm\pi/2$
 and $2 \mid n$, we need to find two more zeros in order to increase the the lower bound we have thus far, i.e. $n-4 $, to the claimed lower bound of $n-2$. Suppose thus that $2 \mid n$. The identity
 \[
H_{n}(c-x)=(-1)^{n}H_{n}(c+x)
\]
implies that if $c$ is a zero of $H_{n}(z)$, then it is
a double zero. On the other hand, if $c$ is not a zero of this polynomial,
then from Remark \ref{rem:signat0} we conclude that the sign of $H_{n}(c)$
is $(-1)^{c+p}$ and consequently 
\[
\lim_{t\rightarrow0}\Arg p(\zeta(t))=(-1)^{c+p}\frac{\pi}{2}.
\]
Moreover, Proposition \ref{prop:asympsmallt} yields that for $t\asymp e^{-\ln^{4}n}/n$,
\[
\Arg p(\zeta(t))=\Arg(i\sin(c+p+1+int))+o(1)=\begin{cases}
o(1) & \text{ if }2\mid c+p\\
\pm\pi+o(1) & \text{ if } 2 \nmid c+p\\
\end{cases}.
\]
This implies that 
\begin{equation}\label{eq:changeargalphapm}
\Delta_{0<t\ll e^{-ln^{4}n}/n}\arg p(\zeta(t))=-\frac{\pi}{2}\text{ or }\frac{3\pi}{2}.
\end{equation}
If the change of argument in \eqref{eq:changeargalphapm} is $3\pi/2$, then we have
at least two zeros of $H_{n}(c\pm int)$ on the range $0<t\ll e^{-ln^{4}n}/n$,
since $\pi H_{n}(c+int)$ is the imaginary part of $p(\zeta(t))$.
If the change of argument in \eqref{eq:changeargalphapm} is $-\pi/2$, we deduce from equation \eqref{eq:changearg} that
\[
\Delta_{0<t<T-c_{2}/n^{2/3}}\arg p(\zeta(t))=-\frac{n\pi}{2}+(c+p)\pi-\frac{3\pi}{4}+C.
\]
Thus, the number of real zeros of $H_{n}(s)$ and those on $c+it$
(except the possible zero at $c$) is at least 
\[
2\left(\frac{n}{2}-2k-1\right)+4k=n-2
\]
when $\alpha=\pi/2$, and at least 
\[
2\left(\frac{n}{2}-2k\right)+2(2k-1)=n-2
\]
when $\alpha=-\pi/2$. This complete the case $\alpha=\pm\pi/2$ when
$n$ is even.

%%%%%%%%%%%%%%%%

If $T=T_{2}$ and $p(\zeta(T_{1}))=0$, then for some $c_{2}\in\mathbb{R}^{+}$
and small $\xi>0$, the number of real zeros of $H_{n}(c+int)$ on
$(0,T)\backslash\{T_{1}\}$ is at least 
\begin{align*}
 & \left\lfloor \frac{\Delta_{e^{-ln^{4}n}/n\ll t<T_{1}-\xi/n}\arg p(\zeta(t))}{\pi}\right\rfloor +\left\lfloor \frac{\Delta_{T_{1}+\xi/n<t<T-c_{2}/n^{2/3}}\arg p(\zeta(t))}{\pi}\right\rfloor \\
\ge & \left\lfloor \frac{\Delta_{e^{-ln^{4}n}/n\ll t<T_{1}-\xi/n}\arg p(\zeta(t))}{\pi}+\frac{\Delta_{T_{1}+\xi/n<t<T-c_{2}/n^{2/3}}\arg p(\zeta(t))}{\pi}\right\rfloor -1.
\end{align*}
Counting $T_{1}$ as an additional zero of $H_{n}(1/2+int)$ on $(0,T)$,
we obtain the same number of zeros of this polynomial as in the case
$p(\zeta(T_{1}))\ne0$. 

\subsection*{Case $T=T_{1}$}

 We conclude from Lemma \pageref{lem:changeargsmallt} that for some
$|C|<\pi/2+o(1)$ and $c_{2}\in\mathbb{R}^{+}$ 
\begin{align*}
\Delta_{e^{-ln^{4}n}/n\ll t<T-c_{2}/n^{2/3}}\arg p(\zeta(t)) & =\frac{1}{2}\Delta\arg_{e^{-ln^{4}n}/n\ll t<T-c_{2}/n^{2/3}}g(\zeta(t))+\frac{|c+p|\pi}{2}+\eta+C\\
 & =\frac{1}{2}\left(\Delta_{\gamma_{1}}\arg g(\zeta)+(c+p+1)\pi\right)+\frac{|c+p|\pi}{2}+\eta+C\\
 & =-\frac{n\pi}{2}+\frac{c+p+|c+p|}{2}\pi+\eta+C.
\end{align*}
We compare the last expression with the one in equation \eqref{eq:changearg} and conclude this case, as well as the proof of Theorem \ref{thm:maintheorem}.

\section{The limiting zero distribution density function}\label{sec:limitdist}
While we have found the zero locus of the cognate sequences under investigation, we can extract further information about the limiting behavior of the zeros in terms of their distribution. We do so by compute the limiting probability density function of the zeros of $H_{n}(c+int)$ on $t\in(0,T)$. To this end, for each $x\in(0,T)$
and $\epsilon>0$, we let $N_{n,\epsilon}(x)$ denote the number of
zeros of $H_{n}(c+int)$ on the interval $t\in(x,x+\epsilon)$. It
follows that the limiting probability density function at $x\in(0,T)$
is given by 
\[
\lim_{\epsilon\rightarrow0}\frac{1}{\epsilon}\lim_{n\rightarrow\infty}\frac{N_{n,\epsilon}(x)}{n}.
\]
We note that for any $x\in(0,T)$ and $x\ne T_{1}$ (if $T=T_{2}$),
\[
p^{2}(\zeta(t))\sim g(\zeta(t))=\frac{2\pi\psi^{2}(\zeta)e^{-2n\phi(\zeta,t)}}{n\phi_{z^{2}}(\zeta,t)}
\]
uniformly on $t\in(x,x+\epsilon)$, and consequently 
\begin{align*}
\Delta\arg_{x<t<x+\epsilon}p(\zeta(t)) & =\frac{1}{2}\Delta\arg_{x<t<x+\epsilon}g(\zeta(t))\\
 & =-n\Delta\Im_{x<t<x+\epsilon}\phi(\zeta,t)+\mathcal{O}(\epsilon).
\end{align*}
It is immediate from the Taylor expansion of $\phi(\zeta,\cdot)$
about $x$ that 
\[
\left.\phi(\zeta,t)\right|_{t=x}^{t=x+\epsilon}=\left.\frac{d\phi(\zeta,t)}{dt}\right|_{t=x}\epsilon+\mathcal{O}(\epsilon^{2}),
\]
where 
\begin{align*}
\left.\frac{d\phi}{dt}\right|_{t=x} & =\left.\frac{\partial\phi}{\partial\zeta}\right|_{t=x}\left.\frac{d\zeta}{dt}\right|_{t=x}+\left.\frac{\partial\phi}{\partial t}\right|_{t=x}\\
 & =\left.\frac{\partial\phi}{\partial t}\right|_{t=x}\\
 & =-i\left(\Log(z_{1}-z)+\Log(z_{2}-z)-\Log(z_{1}+z)-\Log(z_{2}+z)\right).
\end{align*}
Thus, using the fact that 
\[
\pi H_{n}(c+int)=\Im(-i\Re(p(\zeta(t)))),
\]
we conclude that 
\begin{align}
\lim_{\epsilon\rightarrow0}\frac{1}{\epsilon}\lim_{n\rightarrow\infty}\frac{N_{n,\epsilon}(x)}{n} & =\lim_{\epsilon\rightarrow0}\frac{1}{\epsilon}\lim_{n\rightarrow\infty}\frac{\left|\Delta\arg_{x<t<x+\epsilon}p(\zeta(t))\right|}{\pi n}.\nonumber \\
 & =\frac{1}{\pi}\ln\left|\frac{(z_{1}+\zeta(x))(z_{2}+\zeta(x))}{(z_{1}-\zeta(x))(z_{2}-\zeta(x))}\right|.\label{eq:limitmeasure}
\end{align}

In the case $x=T_{1}$ when $T=T_{2}$, we note from the previous
section that the number of zeros of $H_{n}(c+int)$ on $t\in(T_{1},T_{1}+\xi/n)$,
for small $\xi>0$, is $\mathcal{O}(1)$. Since 
\[
\Delta\arg_{T_{1}+\xi/n<t<T_{1}+\epsilon}p(\zeta(t))=\frac{1}{2}\Delta\arg_{T_{1}+\xi/n<t<T_{1}+\epsilon}g(\zeta(t))+C
\]
for $|C|<\pi/2+o(1)$, the same argument above also shows that \eqref{eq:limitmeasure}
holds for $x=T_{1}$ as well.

\begin{figure}[h]
\begin{centering}
\includegraphics[scale=0.5]{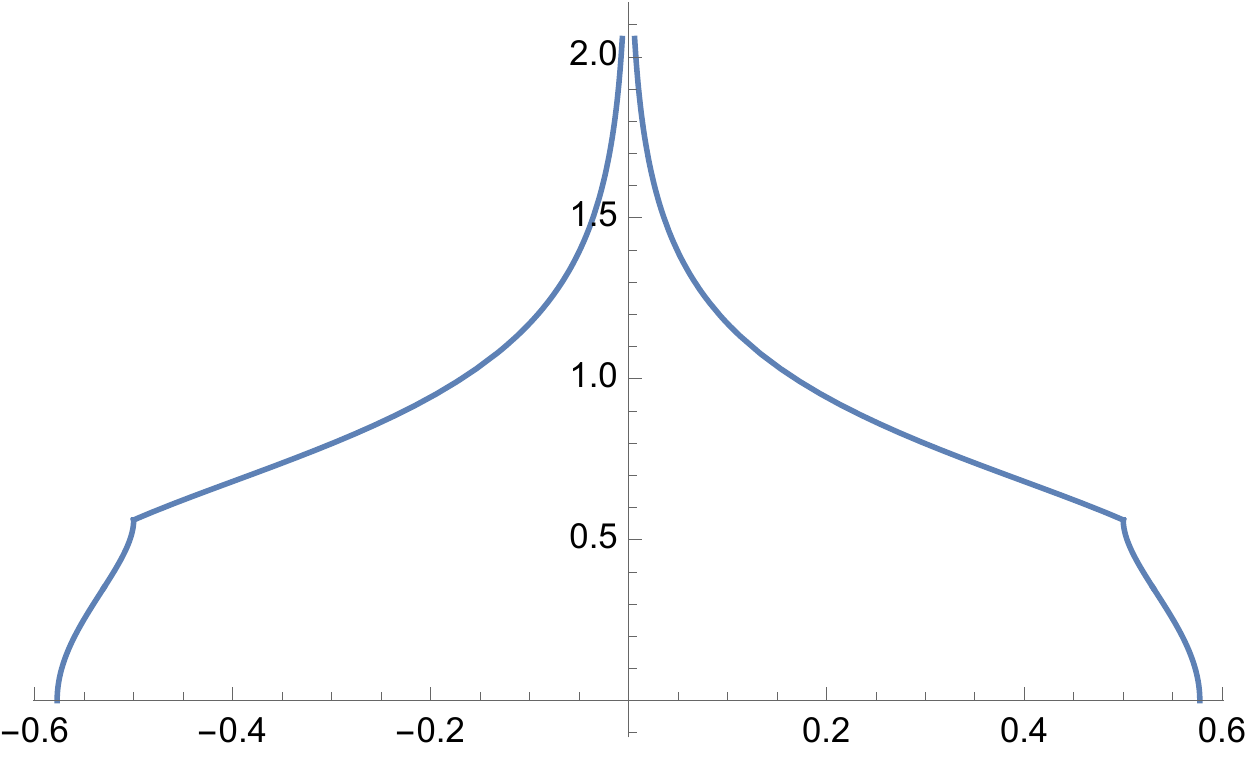} \includegraphics[scale=0.5]{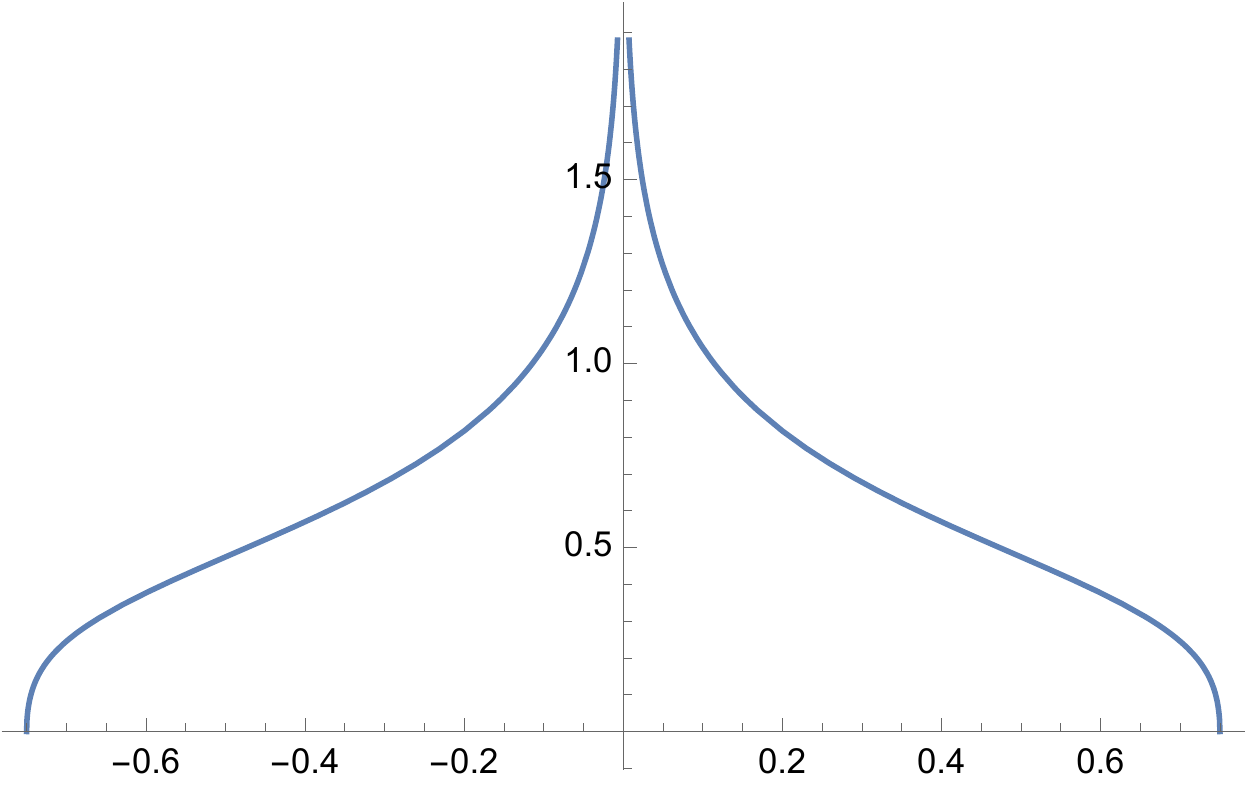} 
\par\end{centering}
\caption{\label{fig:meadist}Limiting probability density function for $(z_{1},z_{2})=(1,3)$
(left) and $(1,7)$ (right)}
\end{figure}

\end{document}